\documentclass[11pt]{amsart}

 \newcommand{\twd}{twisted differential\xspace}

 \usepackage{style}

 \newcommand{\twds}{twisted\ differentials\xspace}

\begin{document}
\title[Curves with a differential with singularities of order $(6;-2)$]{On the locus of genus $3$ curves that admit
meromorphic differentials with a zero of order $6$ and a pole of order $2$}

\author[Castorena]{Abel Castorena}
\email{abel@matmor.unam.mx}
\thanks{Research of the first author is supported by Grants PAPIIT UNAM IN100419 "Aspectos Geometricos del moduli de curvas $M_g$", and CONACyT A1-S-9029 "Moduli de curvas y curvatura en $A_g$.}
\author[Gendron]{Quentin Gendron}
\email{gendron@matmor.unam.mx}
\thanks{Research of the second author was supported by a Posdoctoral Fellowship from DGAPA, UNAM.}

\address{Centro de Ciencias Matem\'aticas-UNAM, Antigua Car. a P\'atzcuaro 8701,
Col. Ex Hacienda San Jos\'e de la Huerta,
Morelia, Mich., M\'exico}

\begin{abstract} 
The main goal of this article is to compute the class of the divisor of $\barmoduli[3]$ obtained by taking the closure of the image of $\omoduli[3](6;-2)$ by the forgetful map. This is done using Porteous formula and the theory of test curves. For this purpose, we study the locus of meromorphic differentials of the second kind, computing the dimension of the map of these loci to $\moduli$ and solving some enumerative problems involving such differentials in low genus. A key tool of the proof is the compactification of strata recently introduced by Bainbridge-Chen-Gendron-Grushevsky-M\"oller.

\end{abstract}
\maketitle
\tableofcontents

\section{Introduction}
\label{sec:intro}
Let $X$ be a smooth projective irreducible complex curve of genus $g$ and let $K_X$ be the canonical line bundle on $X$. The global sections of $K_X$ are the holomorphic differentials, and they form a vector space $H^0(X, K_X)$ of dimension $g$.  A non-zero holomorphic differential~$\omega$ over a curve $X$ induces a translation structure on the complement of the zeroes of~$\omega$ which can be realized as a plane polygon with certain side identifications by translations. Hence the pair $(X,\omega)$ is called indistinctly a {\em translation surface} or an {\em abelian differential} (see for example \cite{mata,mollerTCAG}). 
 \par
Let $\Omega\mathcal M_g$ be the moduli space of abelian differentials $(X,\omega)$ of genus~$g$.
This forms a vector bundle $\Omega\mathcal M_g \to\mathcal M_g$ whose fiber 
on $X\in\mathcal M_g$ is the space $H^0(X,K_X)$, which is  called the {\em Hodge bundle}.
Let $\mu=(a_1,\dots ,a_n)$ a positive partition of $2g-2$, that is, the integers $a_i\in\bZ_{>0}$ 
satisfy
$\sum\limits_{j=1}^n a_{j}=2g-2$. The {\em stratum of abelian differentials $\omoduli(\mu)$ of type $\mu$} parametrises all couples $(X,\omega) \in \omoduli$
with prescribed zeroes of order $a_i$ at distinct points $z_i\in X$ for  $i=1,\dots ,n$. 
\par
The above construction can be extended to the case of meromorphic abelian differentials. For every partition  $\mu = (a_{1},\dots,a_{n};-b_{1},\dots,-b_{p})$ of $2g-2$ with $a_{i},b_{j}\geq1$ there is a moduli space $\omoduli(\mu)$ parametrising the pairs $(X,\omega)$, where $X$ is a genus $g$ curve and $\omega$ a meromorphic differential with zeroes of order $a_{i}$ at some points $z_{i}$ and poles of order~$b_{j}$ at the points~$w_{j}.$
\par
From the projection map $\omoduli\to\moduli[g]$ we have a projective bundle $\PP(\omoduli)$ over $\moduli[g]$ with fibre $\PP^{g-1}$. The image of the stratum $\omoduli(\mu)$ in $\PP(\moduli)$ is called the {\em projective stratum} $\PP\omoduli(\mu)$. Similarily we can define  the projective strata in the meromorphic case. In both cases, the  projective strata $\PP\omoduli(\mu)$ parametrize abelian differentials modulo  multiplication of the form by non-zero complex scalars.
\par
There is a natural map $\pi\colon\omoduli(\mu) \to \moduli$ which factors through $\PP\omoduli(\mu)$ forgetting the differential.  We denote by $\moduli(\mu)$ the image of the stratum $\omoduli[g](\mu)$ in $\moduli$ by~$\pi$ and by $\barmoduli(\mu)$ its closure in the Deligne-Mumford moduli space of stable curves~$\barmoduli$. These loci are interesting  subloci of $\barmoduli$, but only a few results on them are known.  The dimension of (the irreducible components of) $\moduli(\mu)$ have been computed in \cite{bordDeOMg} in the holomorphic case and \cite{buddim} in the meromorphic case. Moreover, in the holomorphic case, when the locus $\barmoduli(\mu)$ is a divisor in $\barmoduli$, its class has been computed in \cite{mudiv}. 
\smallskip
\par
In this paper, we study the loci $\barmoduli[g](\mu)$ in the meromorphic case. 
The main result of this article is the computation of the class of $\barmoduli[3](6;-2)$ in the Picard group  $\Pic(\barmoduli[3])$ of the moduli stack~$\barmoduli[3]$. Recall that this group is generated by the first Chern class $\lambda$ of the Hodge bundle, and by the two boundary divisors $\delta_{0}$ and $\delta_{1}$.

\begin{theorem}\label{thm:classeismenosdos}
 The class of $\barmoduli[3](6;-2)$ in $\Pic(\barmoduli[3])\otimes\bQ$ is
\begin{equation}\label{eq:classeismenosdos}
 \left[ \barmoduli[3](6;-2) \right] = 17108\lambda- 1792\delta_{0} - 4396 \delta_{1}\,.
\end{equation}
\end{theorem}
The computation of the  $\lambda$-class in $\Pic(\moduli[3])\otimes\bQ$ is done in Section~\ref{sec:classelambda} using Porteous formula. For the class in $\Pic(\barmoduli[3])\otimes\bQ$ we use the theory of test curves and degeneration techniques in Section~\ref{sec:class}. This computation has been checked using sage package {\em admcycles} \cite{schmittprog} based on a result of \cite{bhpss} proving a conjecture of Janda, Pandharipande,  Pixton and  Zvonkine. 
\smallskip
\par
In order to perform the test curve in Section~\ref{sec:compu},  some enumerative problems related to abelian differentials naturally appear. In particular, we consider the differentials whose residue at every poles is equal to zero. These differentials are classically called {\em differentials of the second kind}. We denote by $\ores(\mu)$ the locus of $\omoduli(\mu)$ parametrizing the meromorphic differentials of type $\mu$ of the second kind. In Section~\ref{sec:dimproybis} we compute the dimension of the image of $\ores(\mu)$  by the forgetful map inside the moduli space of curves. 
\begin{theorem}\label{thm:dimprojintro}
Let $\mu = (a_{1},\dots,a_{n};-b_{1},\dots,-b_{p})$ be  a partition of $2g-2$ such that $p\geq2$ and $b_{i}\geq2$ for all~$i$. 
\begin{itemize}
 \item If $g=1$ the dimension of the projection of every component of $\ores[1](\mu)$ to $\moduli[1,1]$ forgetting all but one marked point is~$1$.
 \item If $g\geq 2$, the dimension of the projection of $\ores(\mu)$ to $\moduli$ is $\min \left\{ 3g -3 ;2g + n - 2 \right\}$. 
\end{itemize}
\end{theorem}
\par
After giving some general results on families of stable curves and multi-scale differentials in Section~\ref{sec:gendegenfam}, we solve some enumerative problems on differentials of the second kind in  Section~\ref{sec:enum}. 
The most interesting enumerative problem that we solve is the following one.
\begin{theorem}\label{thm:degreeprojs}
   The map $\pi\colon \PP\ores[1](6;-2,-2,-2) \to \moduli[1,1]$ forgetting the polar points is an unramified cover of degree~$7$. 
\end{theorem}

This result can be interpreted in the following way. On a fixed curve $X$ of genus $1$ there exist $7$ differentials in $\ores[1](6;-2,-2,-2)$ modulo translation on $X$ and multiplication by~$\CC^{\ast}$ of the differential.

\smallskip
\par
To conclude, note that according to \cite{boissymero} the strata of $\omoduli[3](6;-2)$ has three connected components. Each component gives rise to an irreducible component of the divisor $\moduli[3](6;-2)$. We make some comments on this problem at the end of this paper (see Corollary~\ref{cor:comp}) and in a future work we want to study each of these irreducible components.

\smallskip
\par
\paragraph{\bf Acknowledgements.} 
We want to thank warmly Johannes Schmitt for detecting errors with the program \cite{schmittprog} and for stimulating discussions leading to the corrected version.
We thank Dawei Chen, Noe B\'arcenas and Scott Mullane  for their valuable help on several mathematical topics.  Moreover we thank Miguel Magaña Lemus and Gerardo Tejero Gómez for technical support.


\section{Background}
\label{sec:rappel}

In this section we recall some known facts, first about the multi-scale differentials  and then about Chern classes of the moduli space.
\smallskip  
\par
 First we give  some common background  for both sections. Let consider the moduli space of smooth $m$-pointed genus $g$ curves $\moduli[g,m]$. Let 
 $\pi\colon\mathcal X\to\moduli[g,m]$ be the universal curve over $\moduli[g,m]$, and let 
 $\omega_{\pi}:=\omega_{\mathcal X|\moduli[g,m]}$ be the relative dualizing sheaf. We have that 
 $\Omega\mathcal M_{g,m}:=\pi_*(\omega_{\pi})$ is a vector bundle over $\moduli[g,m]$, that is the pull-back of the Hodge bundle.

Let $n,p$ be strictly positive integers and we introduce the notation $m = n+p$. We denote by $\mu = (a_1,\dots ,a_n; -b_1,\dots ,-b_p)$ a $m$-tuple of integers such that $a_{i},b_{j}\geq 1$ and $\sum\limits_{j=1}^n a_j - \sum\limits_{i=1}^{p} b_i = 2g-2$. 
We study meromorphic differentials $\eta$ with zeroes of order $a_{i}$ at~$z_{i}$ and poles of order $b_{j}$ at $w_{j}$, i.e. such that $(\eta)=(\eta)_0-(\eta)_{\infty}=\sum a_j z_j-\sum b_i w_i$. We will usually write $\bfz$ for the tuple of points $(z_{1},\dots,z_{n};w_{1},\dots,w_{p})$.
For $j=1,\dots, m$, let 
 $\mathcal D_j$ be the sections of the universal curve corresponding to the marked points $z_j$ if $j\leq n$ and $w_{n+j}$ if $j\geq n+1$.
The strata $\Omega\mathcal M_{g}(\mu)$ of abelian differentials of type $\mu$ is defined to be the subspace of 
$\pi_*(\omega_{\pi})(\sum_{i=1}^p b_i\mathcal D_{n+i})$ of differentials which vanish to order $a_{i}$ at the sections~$\mathcal D_i$ up to the action of the group permuting the singularities of the same order.

\subsection{The multi-scale differentials}
\label{sec:msd}

In this section we recall some notions that we need about twisted and multi-scale differentials as introduced in \cite{IVC} and \cite{complis}.
\par
We begin by recalling the notion of twisted differentials on which is based the more sophisticated notion of multi-scale differentials.
\begin{defn}\label{def:twisteddif}
 A {\em \twd~$\omega$ of type $\mu$ (compatible with a full order $\cleq$)}
on a stable~$n$-pointed curve $(X,\bfz)$ is a collection of
(possibly meromorphic) differentials $\omega_v$ on the irreducible components~$X_v$ of~$X$  and a full order $\cleq$ on the set of these components,
such that no $\omega_v$ is identically zero, with the following properties:
\begin{itemize}
\item[(0)] {\bf (Vanishing as prescribed)} Each differential $\omega_v$ is
holomorphic and non-zero outside of the nodes and marked points of~$X_v$.
Moreover, if a marked point $z_j$ lies on~$X_v$, then $\ord_{z_j} \omega_v=m_j$.
\item[(1)] {\bf (Matching orders)} For any node of~$X$ that identifies
$q_1 \in X_{v_1}$ with $q_2 \in X_{v_2}$,
$$\ord_{q_1} \omega_{v_1}+\ord_{q_2} \omega_{v_2} = -2\,. $$
\item[(2)] {\bf (Matching residues at simple poles)}  If at a node of~$X$
that identifies $q_1 \in X_{v_1}$ with $q_2 \in X_{v_2}$ the condition $\ord_{q_1}\omega_{v_1}=
\ord_{q_2} \omega_{v_2}=-1$ holds, then $\Res_{q_1}\omega_{v_1}+\Res_{q_2}\omega_{v_2}=0$.
\item[(3)]{\bf (Partial order)} If  a node of~$X$  identifies
$q_1 \in X_{v_1}$ with $q_2 \in X_{v_2}$, then $v_1\succcurlyeq  v_2$ if and
only if $\ord_{q_1} \omega_{v_1}\ge -1$. Moreover,  $v_1\asymp v_2$ if and only if
$\ord_{q_1} \omega_{v_1} = -1$.
\item[(4)] {\bf (Global residue condition)} For every level~$i$
and every connected component~$Y$ of $X_{>i}$ that does not
contain a pole $w_{i}$ the following condition holds.
Let $q_1,\dots,q_b$ denote the set of all nodes where~$Y$ intersects $X_{(i)}$. Then
$$ \sum_{j=1}^b\Res_{q_j^-}\omega=0\,,$$
where by definition $q_j^-\in X_{(i)}$.
\end{itemize}
\end{defn}
\par
Note that by point (1) of Definition~\ref{def:twisteddif}, at each node of $X$ the twisted differential $\omega$, either has two simple poles or has a  zero of order $k$ on one branch of the node and a pole of order $-k-2$ on the other branch. The {\em prong number} at a node is $0$ in the first case and $\kappa=k+1$ in the second case. 
\par
Now we give the definition of a multi-scale differential, referring to \cite{complis} for details.
\begin{defn}\label{def:msd}
 A {\em multi-scale differential $(X, \bfz, \omega, \cleq, \sigma)$ of type $\mu$} is a stable pointed curve $(X,\bfz)$ with $\bfz=(z_{1},\dots,z_{n};w_{1},\dots,w_{p})$, a twisted differential $\omega$ of type $\mu$ over $X$ compatible with the total order $\cleq$ and a global prong-matching~$\sigma$.
\end{defn}
\par
The notion of prong-matching is introduced and discussed in great details in \cite{complis}. For us it will not be crucial to know its precise definition. It is sufficient to now that it gives a way to glue the differentials at the nodes. We will use mainly the facts that it is a finite data and that the number of possible classes of prong-matching is computable.  In the  important case of a multi-scale differential with two levels, this number is
\begin{equation}\label{eq:nbpm}
 n  = \gcd (\kappa_{i}) \,,
\end{equation}
where $i$ runs through the set of nodes of the multi-scale differential.

The importance of the notion of multi-scale differentials comes from the following theorem proved in \cite{complis}.
\begin{theorem}\label{thm:complis}
 The moduli space $\LMS$ of the multi-scale differentials of type $\mu$ is a  smooth compactification of the stratum $\omoduli(\mu)$.
\end{theorem}

Moreover, there exists a good system of coordinates near the boundary of this space (see \cite[Section 11]{complis}). The perturbed periods coordinates give a way to understand the families of degenerating differentials with special properties. In this article, the relevant information that we need is that near a given multi-scale differential, the top differential is almost constant, while on the lower levels the differentials are multiplied by~$\prod t_{i}^{a_{i}}$, where $t_{i}$ is a local parameter for each level of the multi-scale differential and $a_{i}$ is an integer defined in \cite[Equation (6.7)]{complis}. Moreover, we know the local equations of the nodes in the universal family. 
\par
In the useful case of a multi-scale differential with two levels we can be more specific. The local equation of the family at the node $n_{i}$ is $x_{i}y_{i} = t^{a_{i}}$ for a local parameter $t$ and with $a_{i} = \lcm(\kappa_{n})/\kappa_{i}$.

\smallskip
\par

\subsection{Chern classes.}
\label{sec:Chern}

The goal of this section is to recall  some facts about Chern classes on the moduli space of curves. 
In order to make this section self contained, we begin by recalling some well-know facts of algebraic geometry.

\subsubsection{Some Notation.}
 Let $f\colon Z\hookrightarrow Y$ be a  a closed immersion of schemes, and denote the sheaf of K\"ahler differentials of $Z$ over $Y$ by $\Omega^1_{Z|Y}$.
 We have the exact sequence
$$0\to I_Z\to\mathcal O_Y\to\mathcal O_Y/I_Z=\mathcal O_Z \to 0 \,,$$
where $I=I_Z$ is the ideal sheaf of~$Z$. Let $F$ be a coherent sheaf on $Y$, then tensoring the above exact sequence by $F$ we get 
$$0\to I_Z\otimes F\to F\to F\otimes\mathcal O_Z=F_Z \to 0\,.$$
 The conormal sheaf  $\mathcal N^{\vee}_{Z|Y}$ of $f$ is the quasi-coherent $\mathcal O_Z$-module $I/I^2$ and the normal sheaf is $\mathcal N_{Z|Y}=\text{Hom}_{\mathcal O_Z}((I/I^2),\mathcal O_Z)$.   Suppose now that $Z$ is an effective Cartier divisor and let $F=\mathcal O_Y(Z)$ the associated invertible sheaf, then we have an exact sequence
$$0\to \mathcal O_Y\to\mathcal O_Y(Z)\to \mathcal O_Z(Z)\to 0 \,,$$
 where $\mathcal O_Z(Z)=\mathcal O_Y(Z)|_Z$. In this case we have that $\mathcal N^{\vee}_{Z|Y}=\mathcal O_Y(-Z)|_Z$ and $\mathcal N_{Z|Y}=\mathcal O_Y(Z)|_Z$. From the exact sequence
$$0\to\mathcal O_Y(-2Z)\to\mathcal O_Y(-Z)\to\mathcal O_Y(-Z)|_Z\to 0 $$
 we have that 
 \begin{equation*}
  \Omega^1_{Z|Y}\simeq I/I^2=\mathcal O_Y(-Z)/\mathcal O_Y(-2Z)=\mathcal O_Y(-Z)|_Z=\mathcal N^{\vee}_{Z|Y}
 \end{equation*}
 and
\begin{equation}\label{eqestrella2}
 (\Omega^1_{Z|Y})^{\vee}\simeq\mathcal O_Y(Z)/\mathcal O_Y=\mathcal O_Y(Z)|_Z=\mathcal N_{Z|Y}\,.
\end{equation}
We recall that if $A$ is a ring and $N\subset M\subset L$ are $A$-modules, then $L/M\simeq (L/N)/(M/N)$. In our case we have the inclusion  $\mathcal O_Y\subset\mathcal O_Y((n-1)Z)\subset \mathcal O_Y(nZ)$ of coherent sheaves on $Y$. Hence we can consider quotient sheaves to get the following exact sequences on~$Y$
\begin{equation}\label{eq3}
 0\to \mathcal O_Y((n-1)Z)/\mathcal O_Y\to\mathcal O_Y(nZ)/\mathcal O_Y\to\mathcal O_Y(nZ)/\mathcal O_Y((n-1)Z)\to 0 \,.
\end{equation}

\subsubsection{The setting.} 
\label{sec:setting}

Given any family $\pi\colon \calX\to B$ of curves of genus $g$, we denote by $\omega_{\calX|B}$ the relative dualizing sheaf of the family~$\pi$. 
When the family $\pi$ contains singular fibers, we have that $\omega_{\calX|B}$ is equal to $\Omega^1_{\calX|B}$ away of the nodes of the fibers, thus,  when the family is a family of smooth curves we can identify $\omega_{\calX|B}\simeq\Omega^1_{\calX|B}$.

Consider the universal curve $\pi_0\colon\mathcal X=\mathcal M_{g,1}\to\mathcal M_g$. For $g\geq 2$ this map is smooth of relative dimension one. We denote by 
$\Omega$ the relative dualizing sheaf $\omega_{\mathcal X|\mathcal M_g}$ associated to~$\pi_0$. 
Let $\pi^{n}\colon\mathcal X^n \to \moduli$ be the $n$-fold fiber product of $\mathcal X$ over 
$\mathcal M_g$. The space $\mathcal X^n$ parametrises smooth genus $g$ curves with $n$-tuples of not necessary distinct points.  Note that the fiber  over $X\in\mathcal M_g$ of $\pi^{n}$ is the direct product $X^n= X\times\dots\times X$ and  the fiber of $\mathcal M_{g,n} \to \moduli$ is the complement of the diagonal $\Delta$ in $X^n$, where 
$$\Delta=\{(X, q_1;\dots,q_n): q_i=q_j\text{ for at least two indices }i\neq j\}\,.$$
Let $\Delta_{ij}$ be the diagonal corresponding the points where $q_i=q_j$ for two indices $i\neq j$. Let $\pi_i\colon\mathcal X^n\to\mathcal X$ be the forgetful map which forgets all but the $i$-th factor and let define the sheaf $\Omega_i=\pi_i^*(\Omega)$. 

\subsubsection{The Chern classes for $\pi_*(\Omega_1(2\Delta))$.} 
We restrict our attention to the case $n=1$, that is, we consider the projection $\pi=\pi_1\colon\mathcal X^2\to\mathcal X$ on the first factor. We will write 
$\mathcal O=\mathcal O_{\mathcal X^2}$ and consider the diagonal of the fiber product on 
$\mathcal X^2$, that is, $\Delta\colon\mathcal X\to\mathcal X^2=\mathcal X\times_{\mathcal M_g}\mathcal X$, and we identify 
$\Delta(\mathcal X)\simeq\mathcal X$. We write $\Delta$ to denote the image $\Delta(\mathcal X)$. From identifications in Equation~\eqref{eqestrella2} and from the exact sequence~\eqref{eq3}  with $Z=\Delta$ and $Y=\mathcal X^2$, we have that the normal bundle $\mathcal N_{\Delta|\mathcal X^2}$ satisfies  
\begin{equation} \label{eq4}
 \mathcal N_{\Delta|\mathcal X^2}=\mathcal O_{\mathcal X^2}(\Delta)/\mathcal O_{\mathcal X^2}\simeq(\Omega^1_{\Delta|\mathcal X^2})^{\vee}\simeq(\pi_1^*(\Omega^1_{\mathcal X|\mathcal M_g}))^{\vee}=(\Omega_1)^{\vee}=\pi_1^*(\Omega)^{\vee}\,.
\end{equation}
With this notation, we twist the sequence \eqref{eq3} for $n=2$ by $\Omega_1=\pi_1^*(\Omega)$ to get the exact sequence
 \begin{equation*}
  0\to\Omega_1\otimes(\Omega^1_{\Delta|\mathcal X^2})^{\vee}\to\Omega_1\otimes(\mathcal O(2\Delta)/\mathcal O)\to\Omega_1\otimes((\Omega^1_{\Delta|\mathcal X^2})^2)^{\vee}\to 0\,.
 \end{equation*}

 Using Equation~\eqref{eq4}  the previous exact sequence reads
\begin{equation*}
 0\to\mathcal O_{\mathcal X^2}\to\Omega_1\otimes(\mathcal O(2\Delta)/\mathcal O)\to(\Omega_1)^{\vee}\to 0 \,.
\end{equation*}
 Pushing down this exact sequence to $\mathcal X$ with $\pi_*$ we have the following exact sequence coherent sheaves
\begin{equation}\label{eq:evF}
 0\to F_0\to F_1\to F_2\to 0\, ,
\end{equation}
where $F_0:=\pi_*(\mathcal O_{\mathcal X^2})=\mathcal O_{\mathcal X}$,  $F_1:=\pi_*(\Omega_1\otimes(\mathcal O(2\Delta)/\mathcal O))$ and $F_2:=\pi_*(\Omega_1^{\vee})$. 
The fiber of $F_1$ at a point $(X,p)\in\mathcal X$ is the  two-dimensional vector space of sections $H^0(X, K_X(2p)/K_X)$. Similarly the fibers of $F_0$  and $F_2$ at $(X,p,q)$ are respectively  the vector spaces $H^0(X,\mathcal O_X)$ and $H^0(X,K_X^{\vee})=\{0\}$. So the exact sequence \eqref{eq:evF} reads
$$0\to\mathcal O_{\mathcal X} \to F_1\to\pi_*(\Omega_1^{\vee})\to 0 \,,$$
which implies that the Chern classes satisfy
\begin{equation}\label{eq7}
 c(F_1)=c(\mathcal O_{\mathcal X^2})c(\pi_*(\Omega_1^{\vee}))=1-K_1 \,,
\end{equation}
  where, following the notation in \cite{faber}, we have that $K_1=c_1(\pi_*(\Omega_1^{\vee}))$ is the first Chern class. 

 Let $E=\pi_*(\Omega_1)$ and $E_2=\pi_*(\Omega_1\otimes\mathcal O_{\mathcal X^2}(2\Delta))=\pi_*(\Omega_1(2\Delta))$. The exact sequence 
$$0\to E\to E_2\to F_1\to 0 $$
implies that the Chern classes satisfy $c(E_2)=c(E)c(F_1)=(1+\lambda)(1-K_1)$. 
This leads to the following result.
\begin{prop}\label{prop:chern2p}
The Chern class of $E_2$ is
 \begin{equation*}
  c(E_2)= 1 + \sum_{i\geq1} (\lambda_{i}-\lambda_{i-1}K_1)\,.
 \end{equation*}
\end{prop}
To conclude, let us remark that we can define a similar vector bundle on $\calX^{n}$. It suffices to consider  the exact sequence 
\begin{equation*}
 0\to\mathcal O_{\mathcal X^n}\to\Omega_1\otimes(\mathcal O(2\Delta_{1,n})/\mathcal O)\to(\Omega_1)^{\vee}\to 0 
\end{equation*}
and to define $E_{2}$ to be the push-forward by the map forgetting the last point of the middle term. Then Proposition~\ref{prop:chern2p} remains true in this generalised context.
\smallskip
\par
\subsubsection{The Chern classes for $\pi_*(\Omega_1(n\Delta))$ with $n\geq3$.} Now we want to extend the result of the formula of the Chern class of Proposition~\ref{prop:chern2p} to the bundle $E_n:=\pi_*(\Omega_1(n\Delta))$ on $\mathcal X$ for all $n\geq 3$. Note that the fiber of $E_n$   at a  point $(X,p)\in\mathcal X$ is the space $H^0(X,K_X(np))$  of differentials on $X$ that have at worst poles of order $n$ at~$p$. 
\par
By twisting the exact sequence \eqref{eq3} by $\Omega_1$ and using Equation~\eqref{eq4} we obtain the exact sequence
$$0\to\pi_*(\Omega_1\otimes\mathcal O((n-1)\Delta)/\mathcal O))\to\pi_*(\Omega_1\otimes\mathcal O(n\Delta)/\mathcal O))\to\pi_*((\Omega_1^{\otimes n-1})^{\vee})\to 0 $$
on $\mathcal X$.
This gives a recursive formula for Chern classes of $E_n$ as follow
\begin{equation*}
 c(\pi_*(\Omega_1\otimes\mathcal O(n\Delta)/\mathcal O)))=c(\pi_*(\Omega_1\otimes\mathcal O((n-1)\Delta)/\mathcal O))c(\pi_*((\Omega_1^{\otimes n-1})^{\vee}))\,.
\end{equation*}
Since the base case $ c(\pi_*(\Omega_1\otimes\mathcal O(2\Delta)/\mathcal O)))=(1-K_1)$   is given by equation~\eqref{eq7}, we obtain 
\begin{equation}\label{eq:clasen}
 c(\pi_*(\Omega_1\otimes\mathcal O(n\Delta)/\mathcal O)))=\prod _{i=1}^{n-1} (1-iK_1)\,.
\end{equation}
\smallskip
\par

 Let $X$ be a smooth curve and $p\in X$. For all positive integer $n$ we have the following exact sequence
\begin{equation*}
 0\to K_X\to K_X(np)\to K_X(np)|_{np}\to 0 \,.
\end{equation*}
In order to globalize this exact sequence to the universal curve $\mathcal X$ we proceed as follow. On~$\mathcal X^2$ we have the following exact sequence
\begin{equation}\label{eq10}
 0\to\mathcal O\to\mathcal O(n\Delta)\to(\mathcal O(n\Delta))|_{n\Delta} \to 0\,.
\end{equation}
By tensoring with $\Omega_1$ and pushing down to~$\mathcal X$ we obtain the following exact sequence
$$
0\to\pi_*(\Omega_1)\to\pi_*(\Omega_1(n\Delta))\to\pi_*(\Omega_1(n\Delta)|_{n\Delta})\to 0\,.$$
Since the exact sequence~\eqref{eq10} gives that $\mathcal O_{\mathcal X^2}(n\Delta)/\mathcal O_{\mathcal X^2}=\mathcal O(n\Delta)|_{n\Delta}$, we identify the vector bundle $\pi_*(\Omega_1(n\Delta)|_{n\Delta})$ with $\pi_*(\Omega_1\otimes(\mathcal O(n\Delta)/\mathcal O))))$. Hence Equation~\eqref{eq:clasen} gives the class $c(E_n)$.
 \begin{prop}\label{prop:chernnp}
  The Chern class of $E_n$ is
 \begin{equation*}
  c(E_n)= (1 + \lambda)\prod _{i=1}^{n-1} (1-iK_1)\,.
 \end{equation*}
 \end{prop}
Finally note that, as in the case of $E_{2}$, this proposition can be extended on $\calX^{n}$ to the sheaf $E_{n}$   similarly defined.

 \smallskip
 \par

 \subsubsection{Known facts about Chern Classes.}\label{sec:recallchern}
 
 We conclude this section by recalling two known facts about Chern classes. The first one is the inversion formula for Chern classes. The second is some equalities for the Chern classes above the moduli space.
 \par
 We first give a formula in order to compute the inverse of a Chern class. Much more material around this circle of ideas can be found in \cite{hirzebruch}.
 Let us first define the polynomial 
 \begin{equation*}
 P_{n} (x_{1},\dots,x_{n}) = \sum_{i_{1}+2i_{2}+\cdots+ni_{n} = n}\left( \frac{(i_{1}+\cdots+i_{n})!}{i_{1}!\cdots i_{n}!}\prod_{j=1}^{n}(-x_{j})^{i_{j}} \right) \,. 
 \end{equation*}
To be concrete, the polynomials $P_{n}$ for $n\leq3$  are 
\begin{eqnarray*}
P_{0} &=& 1 \,, \\
P_{1} &=& -x_{1} \,, \\
P_{2} &=& x_{1}^{2}-x_{2} \,, \\
P_{3} &=& -x_{1}^{3} + 2x_{1}x_{2} - x_{3} \,.
\end{eqnarray*}
The importance of these polynomials is given by the following result that will be used several times in Section~\ref{sec:class}.
\begin{lem}\label{lem:inversechern}
 Let $E$ be a complex vector bundle over a complex manifold~$X$ whose  Chern class is $c(E)=1+c_{1}(E)+ c_{2}(E) + \cdots + c_{n}(E)$. Then the Chern class of $-E$ is 
 \[c(-E)=1+P_{1}(c_{1})+P_{2}(c_{1},c_{2})+\cdots+ P_{n}(c_{1},\dots,c_{n})\,.\]
\end{lem}
\par
We now recall some results about Chern classes on the moduli space. These results are due to \cite{HM} and many examples of application can be find in \cite{faber}.
\begin{lem}\label{lem:formularium}
 Let $K_{i}$ be the class of $\Omega_{i} = \pi_i^*(\Omega)$ and $\Delta_{ij}$ be (the class of) the  $(i,j)$-th diagonal as introduced in Section~\ref{sec:setting}, then
 \begin{eqnarray*}
 \Delta_{id}\Delta_{jd} &=&\Delta_{ij}\Delta_{id} \text{ for } i<j<d\,,\\
 \Delta_{ij}^{2} &=&-K_{i}\Delta_{ij} \text{ for } i<j\,,\\
 K_{j}\Delta_{ij} &=&K_{i}\Delta_{ij} \text{ for } i<j\,.
 \end{eqnarray*}
Moreover denoting by $\pi\colon\mathcal{X}^{d-1} \to \moduli$ the forgetful map, then for every monomial $M$ pulled back from $\mathcal X^{d-1}$
 \begin{eqnarray*}
 \pi_{d,\ast}(M\cdot\Delta_{id}) &=& M\,,\\
 \pi_{d,\ast}(M\cdot K_{d}^{k}) &=& M \cdot \pi^{\ast}(\kappa_{k-1})\,.
 \end{eqnarray*}
\end{lem}

\section{The locus of differentials of the second kind}
\label{sec:dimproybis}

The objective of this section is to compute the dimension of the projection to $\moduli$ of the locus parametrizing the differentials of second kind in the strata of meromorphic differentials. Recall that the differentials of second kind are meromorphic differentials such that the residue at every pole is zero. We denote the locus of differentials of type~$\mu$ of second kind of type $\mu$ by $\ores(\mu)$.
\smallskip
\par
We begin with a preliminary result.
\begin{lem}\label{lm:dimres}
Given a stratum of meromorphic differentials $\omoduli(a_{1},\dots,a_{n};-b_{1},\dots,-b_{p})$ such that $g\geq1$, $p\geq2$ and $b_{i}\neq 1$. The subspace which parametrises differentials with zero residues at the first $i\leq p-1$ points is of codimension~$i$. 
\end{lem}

\begin{proof}
By \cite{geta} these subspaces are not empty. Moreover, each condition on the residues gives a linear equation  in period coordinates. Since these equations involve at most $p-1$ residues, the $i$ equations are independent. Hence the locus that they define is of codimension~$i$.
\end{proof}

We can now compute the dimension of the projection of the loci of differentials of second kind to $\moduli$, proving Theorem~\ref{thm:dimprojintro}.  Recall that this theorem says that the dimension of the projection is $1$ in the genus one case and $\min \left\{ 3g -3 ;2g + n - 2 \right\}$ in the genus $g\geq2$ case.
The proof is by degeneration in the spirit of Theorem~1.3 of \cite{bordDeOMg}. In this proof, we use the notation $b=\sum_{i=1}^{p}b_{i}$.

\begin{proof}
For $g=1$, if there were a component of $\ores[1](\mu)$ such that the dimension of the projection~$0$, then this implies that there exists  a curve $X$ having a one dimensional family of differentials of the second kind of  type $\mu$ on~$X$. We know (see for example Section~2 of \cite{hurwitz}) that a differential of type $\mu$ on $X$ can be written 
\[ \omega = \frac{\sigma^{a_{1}}(z-z_{1})\dots\sigma^{a_{n}}(z-z_{n})}{\sigma^{b_{1}}(z-w_{1})\dots\sigma^{b_{n}}(w-z_{n})} dz \, ,\]
where $\sigma$ is \Weierstrass sigma-fonction of~$X$. Hence the residue is a rational function of the $z_{i}$ and $w_{j}$ which is clearly non constant.
\smallskip
\par
For $g\geq 2$, we degenerate to the curve $X$ pictured in Figure~\ref{fig:CourbeDegRes}. The curve $X_{0}$ is of genus $g-1$ and the genus of $X_{1}$ is~$1$. We denote by $p_{0}$ and $p_{1}$ the nodal points belonging respectively to $X_{0}$ and $X_{1}$.
 \par
    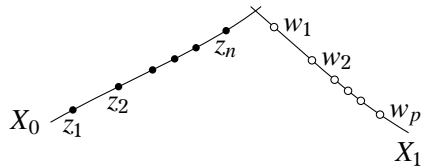
\begin{figure}[ht]
 \centering
\begin{tikzpicture}[scale=1.4]

\draw (-2.3,-1) coordinate (x1) .. controls (-1.6,-.6) and (-.8,-.3) .. (-.3,.1) coordinate (q1) coordinate [pos=.1] (z1) coordinate [pos=.3] (z2)coordinate [pos=.45] (z6)coordinate [pos=.55] (z3) coordinate [pos=.65] (z4) coordinate [pos=.8] (z5);
\foreach \i in {1,2,...,6}
\fill (z\i) circle (1pt);
\node[below] at (z1) {$z_{1}$};
\node[below] at (z2) {$z_{2}$};
\node[below] at (z5) {$z_{n}$};

\node[left] at (x1) {$X_{0}$}; 
\draw (-.4,.1) coordinate (x1) .. controls (.4,-.6) and (.6,-.8) .. (1.1,-1.1) coordinate (q2) coordinate [pos=.1] (z1) coordinate [pos=.3] (z2)coordinate [pos=.45] (z6)coordinate [pos=.55] (z3) coordinate [pos=.65] (z4) coordinate [pos=.8] (z5);
\foreach \i in {1,2,...,6}
\filldraw[fill=white] (z\i) circle (1pt);
\node[ right] at (z1) {$w_{1}$};
\node[ right] at (z2) {$w_{2}$};
\node[ right] at (z5) {$w_{p}$};
\node[ below ] at (q2) {$X_{1}$};
\end{tikzpicture}
\caption{The pointed curve $X$ we are degenerated to.}\label{fig:CourbeDegRes}
\end{figure}

\par
We consider the twisted differential $\omega$ on $X$ such that the restriction $\omega_{i}$ to the irreducible component $X_{i}$ are the following differentials. On $X_{0}$ the differential is in the stratum $\omoduli[g-1](a_{1},\dots,a_{n};-b-2)$ and on $X_{1}$ the differential is in $\ores[1](b;-b_{1},\dots,-b_{p})$. Note that by \cite[Theorem 1.1]{muHur} this twisted differential is in the closure of $\ores(\mu)$. 
Moreover by \cite{buddim} the dimension of the projection $\omoduli[g-1](a_{1},\dots,a_{n};-b-2)$ to $\moduli[g-1]$ is $\min(3(g-1)-3,2(g-1)-2+n)$.
\par
Suppose that $n\geq g-1$, then  there exist a dense subset $U$ of $\barmoduli[g-1]$ such that for every point $X'_{0}$ in $U$ there exists a differential $\omega_{0}$ on~$X'_{0}$. Note moreover that there is a positive dimensional family of such differentials with the polar point $p_{0}$ moving on $X'_{0}$.
By the case of genus $g=1$, the dimension of the projection of $\ores[1](b;-b_{1},\dots,-b_{p})$ is equal to one. Moreover, there is one smoothing parameter at the node.
Summing up the contribution, the dimension of the projection is $3(g-1)-3 + 1 +1 +1 = 3g-3$, where $3(g-1)-3$ is the dimension of the projection of $\omoduli[g-1](a_{1},\dots,a_{n};-b-2)$ to $\moduli[g-1]$ and the three $1$ are respectively the moving point $p_{0} \in X'_{0}$, the dimension of the space of elliptic curves $\left\{(X_{1},p_{1})\right\}$ and the smoothing parameter of the node. 
\par
If $n < g-1$ then by \cite{buddim} there exists a  $2(g-1)-2+n$ dimension space of curves $X_{0}$ which admits such differential $\omega_{0}$. Moreover the set of possible point $p_{0}$ is finite on $X_{0}$. Hence the dimension of the projection of the locus $\ores(\mu)$ is $2(g-1)-2+n +1 +1 =2g-2+n$ in this case. In this sum, the contribution $2(g-1)-2+n$ is the dimension of the projection of 
$\omoduli[g-1](a_{1},\dots,a_{n};-b-2)$ to $\moduli[g-1]$ and the two $1$ are respectively the dimension of the space of elliptic curves $ \left\{ (X_{1},p_{1}) \right\}$ and the smoothing parameter of the node.
\end{proof}

\section{The class of the divisor $\mathcal{M}_{3}(6;-2)$} 
\label{sec:classelambda}

The goal of this section is to compute the class of the projection $\moduli[3](6;-2)$ of  $\omoduli[3](6;-2)$ to $\moduli[3]$. This gives the coefficient of  the $\lambda$-class in Equation~\eqref{eq:classeismenosdos} of Theorem~\ref{thm:classeismenosdos}. In order to do this, we use Porteus Formula and the results recalled in Section~\ref{sec:Chern}.

\smallskip
\par
For $n\geq 1$, set $\mathcal O=\mathcal O_{\mathcal X^n}$. 
Recall that the diagonal $\Delta_{ij}$ is given by the locus where $q_{i}=q_{j}$ in $\mathcal X^{3}$ and define the divisor  
$\mathfrak{D}_{2}^{6} = 2\Delta_{23}-6\Delta_{13}$  inside 
$\mathcal X^3$. Tensoring the exact sequence
\begin{equation*}
 0\to\mathcal O(-6\Delta_{13})\to \mathcal O\to\mathcal O|_{6\Delta_{13}}\to 0 
\end{equation*}
by $\Omega_3(2\Delta_{23})$ we obtain
\begin{equation*}
 0\to\Omega_3\otimes\mathcal O(-6\Delta_{13}+2\Delta_{23})\to\Omega_3\otimes\mathcal O(2\Delta_{23})\to\Omega_3\otimes\mathcal O(2\Delta_{23})|_{6\Delta_{13}}\to 0\,.
\end{equation*}
Pushing down this exact sequence  on $\mathcal X^2$ by the map $\pi$ forgetting the third point, we obtain
$$0\to\pi_*(\Omega_3(\mathfrak{D}_{2}^{6}))\to\pi_*(\Omega_3\otimes\mathcal O(2\Delta_{23}))\to\pi_*(\Omega_3(2\Delta_{23})|_{6\Delta_{13}})\to 0 \,.$$
 We define the following coherent sheaf on $\mathcal X^2$
$$\mathcal F=\pi_*(\Omega_3\otimes(\mathcal O(2\Delta_{23})/(\mathcal O(\mathfrak{D}_{2}^{6}))) \, .$$ 
Note that the stalk $\mathcal F_x$ of $\mathcal F$ at a point $x=(X;z,w)\in\mathcal X^2$ is given by
$$\mathcal F_x=H^0(X,K_X(2w)/K_X(2w-6z))\,.$$
\par
The evaluation map gives a morphism $\phi\colon E_2\to\mathcal F$ over $\mathcal X^2$ such that on the stalk we have
$$ \phi_x \colon (E_2)_x=H^0(X,K_X(2w))\to\mathcal F_x=H^0(X,K_X(2w)/K_X(2w-6z))\,.$$
Since the dimension of the source is $4$ and the dimension of the goal is $6$, by Porteous formula (see \cite[Theorem 3.114]{hamo}) the class of the degeneracy locus where the map has rank less or equal to three is $c_{3}(\mathcal{F}-E_2)$. Note that this degeneracy locus can contain extra components in the diagonal $\Delta_{1,2} = \left\{ (w,z): w=z \right\}$. This is indeed the case and we will deal with this problem at the end of this section.
\par
The Chern classes of $\mathcal F$ is given by the following general formula.
\begin{lem}\label{lm:classF}
 The Chern class of the vector bundle whose fiber is
 \[ H^{0} \left( K\left(\sum_{i=1}^{p} b_{i}w_{i}\right)/K\left(\sum_{i=1}^{p} b_{i}w_{i}-\sum_{i=1}^{n} a_{i}z_{i} \right)  \right)\]
 is equal to
 \begin{equation*}
  \prod_{i=1}^{n}\prod_{j=1}^{a_{i}} (1+jK_{p+i}+\sum_{k=1}^{p} b_{k}\Delta_{k,p+i}) \,.
 \end{equation*}
\end{lem}

\begin{proof}
Given an effective divisor $D$ on a smooth curve $X$, 
the sheaf $K_X(D)$ is the sheaf of meromorphic 1-forms $\omega$ such that for a point $q$ on $X$,  
$(\text{div }\omega)(q)+D(q)\geq 0$, in particular if $q\notin\text{supp}(D)$, then $\omega$ is holomorphic at $q$. If $q$ is a zero of order $a$ for $\omega$, then in an open neighbourhood $U$ around $q$ with local coordinate $z$, we can trivialize $K_X$ so that the form $z^a\cdot\omega$ generates the stalk $K_X(D)|_q$, then $K_X(D)|_q\simeq K_X^{\otimes a}|_q\otimes\mathcal O_q(D)$. 
\par
Recall that we denote by $\mu=(a_1,\dots,a_n;-b_1,\dots,-b_p)$ a partition of $2g-2$  where $a_{i},b_{j}\geq 1$ and that $m=n+p$. 
\par 
On $\mathcal X^m$ we denote $\Delta_Z=\sum_{i=1}^n a_{i}\Delta_{p+i,m+1}$ and $ \Delta_W=\sum_{j=1}^p b_{j}\Delta_{j,m+1}$ the divisor of zeros and poles respectively. We then consider the divisor $ \Delta_{Z,W}=\Delta_Z-\Delta_W$
and the following two exact sequences on $\mathcal X^{m+1}$ and $\mathcal X^m$ respectively:
$$0\to\mathcal O(\Delta_W-\Delta_Z)\to\mathcal O(\Delta_W)\to\mathcal O(\Delta_W)|_{\Delta_Z}\to 0 \, ,$$
$$0\to\pi_*(\Omega_{m+1}(\mathcal O(\Delta_W-\Delta_Z)))\to\pi_*(\Omega_{m+1}(\mathcal O(\Delta_W)))\to\pi_*(\Omega_{m+1}(\mathcal O(\Delta_W)|_{\Delta_Z}))\to 0\,.$$
 We recall that $\mathcal O(\Delta_W)|_{\Delta_Z}=\mathcal O_{\Delta_Z}(\Delta_W)=\mathcal O(\Delta_W)\otimes\mathcal O_{\Delta_Z}$. We set 
$$\mathcal F:=\pi_*(\Omega_{m+1}(\mathcal O(\Delta_W)|_{\Delta_Z}))=\pi_*(\Omega_{m+1}(\mathcal O(\Delta_W))/\mathcal O(\Delta_W-\Delta_Z))\simeq\pi_*(\Omega_{m+1}(\Delta_W)/\Omega_{m+1}(\Delta_W-\Delta_Z))\,.$$
 To compute the Chern classes of $\mathcal F$ we adapt a classical argument that we learned in  \cite{ChenTeich}. 
 For every $i=1,\dots,n$, we consider  the bundle $\Omega_{m+1}(\Delta_W-a_{i}\Delta_{p+i,m+1})=\Omega_{m+1}\otimes\mathcal O(\Delta_W-a_{i}\Delta_{p+i,m+1})$  on $\mathcal X^{m+1}$, where $\Omega_{m+1}$ is Hodge bundle on $\mathcal X^{m+1}$. 
\smallskip
\par
Consider the case $i=1$ and in order to simplify the notation we will use $\varDelta$ for the diagonal $\Delta_{p+1,m+1}$.  We have exact sequences:
$$0\to\Omega_{m+1}(\Delta_W-a_1\varDelta)\to\Omega_{m+1}(\Delta_W-(a_1-1)\varDelta)\to\Omega_{m+1}(\Delta_W-a_1\varDelta)|_{\varDelta}\to 0\,,$$
$$0\to\Omega_{m+1}(\Delta_W-a_1\varDelta)\to\Omega_{m+1}(\Delta_W)\to\Omega_{m+1}(\Delta_W)|_{a_1\varDelta}\to 0 \,,$$
$$0\to\Omega_{m+1}(\Delta_W-(a_1-1)\varDelta)\to\Omega_{m+1}(\Delta_W)\to\Omega_{m+1}(\Delta_W)|_{(a_1-1)\varDelta}\to 0\,.$$

We define 
\begin{eqnarray*}
 & F_1:=\pi_*\left(\frac{\Omega_{m+1}(\Delta_W-(a_1-1)\varDelta)}{\Omega_{m+1}(\Delta_W-a_1\varDelta)}\right),
 \quad F_0:=\pi_*\left(\frac{\Omega_{m+1}(\Delta_W)}{\Omega_{m+1}(\Delta_W-a_1\varDelta)}\right),\\
 &  \text{ and } F_2:=\pi_*\left(\frac{\Omega_{m+1}(\Delta_W)}{\Omega_{m+1}(\Delta_W-(a_1-1)\varDelta)}\right)\,.
\end{eqnarray*}
From the natural inclusions 
$$\Omega_{m+1}(\Delta_W-a_1\varDelta)\subset\Omega_{m+1}(\Delta_W-a_1\varDelta)\subset\Omega_{m+1}(\Delta_W-(a_1-1)\varDelta)\subset\Omega_{m+1}(\Delta_W)$$ 
we have the exact sequence 
$$0\to F_1\to F_0\to F_2\to 0\,.$$
Since the divisor $\varDelta$ has multiplicity $a_1$ at $\Delta_Z$ and $\varDelta\notin\text{Supp}(\Delta_W)$, the fiber of~$F_1$ at $z_{1}$ is 
$$\left(K_X\left(\sum_{j=1}^p b_{j}w_j \right)\right)|_{z_1}\simeq (K_X|_{z_1})^{\otimes a_1}\otimes\mathcal O_{z_1}\left(\sum_{j=1}^p b_{j}w_j \right)\,,$$ 
then we have that 
$F_1\simeq \Omega_n^{\otimes a_1}\otimes\mathcal O_{\mathcal X^n}(\Delta_W)$. 
\smallskip
\par
Following this argument we can to construct a new filtration 
$$0\to F'_1\to F_1\to F'_2\to 0 \,,$$ 
 where $F'_1$ is the sheaf defined as  
 $$F'_1:=\pi_*\left(\frac{\Omega_{m+1}\left(\Delta_W-(m_2-1)\varDelta\right)}{\Omega_{m+1}\left(\Delta_W-(a_1-1)\varDelta\right)} \right)\,,$$ 
 and we have that $F'_1\simeq\Omega_n^{a_1-1}\otimes\mathcal O_{\mathcal X^n}(\Delta_W)$. In this way using a sequence of filtrations obtained by subtracting~$1$ successively to $a_1$ we get that 
$$F_0\simeq[\Omega^{\otimes a_1}_n\otimes\mathcal O_{\mathcal X^n}(\Delta_W)]\otimes[\Omega^{\otimes (a_1-1)}_n\otimes\mathcal O_{\mathcal X^n}(\Delta_W)] \otimes\dots\otimes[\Omega_n\otimes\mathcal O_{\mathcal X^n}(\Delta_W)] \, .$$ 
By subtracting $1$ successively to all $a_{i}$ we can reduce to signature $(1,0,\dots,0; -b_1,\dots,-b_p)$ to obtain the expression for the Chern class of $\mathcal F$ as desired.
\end{proof}

Applying Lemma~\ref{lm:classF} to the case of the stratum $\omoduli[3](6;-2)$, we get that
\begin{eqnarray*}
 c(\mathcal{F}) &=&  \prod_{j=1}^{6} (1+jK_{2}+2\Delta_{12})\, ,\\
   &=& 1+ (21K_{2}+12\Delta_{12}) +(175 K_{2}^{2}+210K_{2}\Delta_{12}+60\Delta_{12}^{2})\\
   & & + (735K_{2}^{3}+1400K_{2}^{2}\Delta_{12}+840 K_{2}\Delta_{12}^{2}+160 \Delta_{12}^{3}) + \dots \,.
\end{eqnarray*}
By Proposition~\ref{prop:chern2p} we obtain that the class of $-E_2$ is 
\begin{eqnarray*}
 c(-E_2) &=& 1+(K_{1}-\lambda_{1}) +(K_{1}^{2}-\lambda_{1}K_{1} +\lambda_{1}^{2}-\lambda_{2}) \\
   & & + (K_{1}^{3}-\lambda_{1}K_{1}^{2} +(\lambda_{1}^{2}-\lambda_{2})K_{1} +2\lambda_{1}\lambda_{2}  -\lambda_{3} -\lambda_{1}^{3} ) + \dots\,.
\end{eqnarray*}
So we obtain 
\begin{eqnarray*}
 c_{3}(\mathcal{F}-E_2) &=& (735K_{2}^{3}+1400K_{2}^{2}\Delta_{12}+840 K_{2}\Delta_{12}^{2}+160 \Delta_{12}^{3})\\
 & & + (175 K_{2}^{2}+210K_{2}\Delta_{12}+60\Delta_{12}^{2})\cdot (K_{1}-\lambda_{1})\\
  & & + (21K_{2}+12\Delta_{12})\cdot (K_{1}^{2}-\lambda_{1}K_{1} +\lambda_{1}^{2}-\lambda_{2}) \\
 & & + (K_{1}^{3}-\lambda_{1}K_{1}^{2} +(\lambda_{1}^{2}-\lambda_{2})K_{1} +2\lambda_{1}\lambda_{2}  -\lambda_{3} -\lambda_{1}^{3} )\,.
\end{eqnarray*}
This is equal to the Chern class of $\mathcal{M}_{3}(6;-2)$ inside $A^{3}(\mathcal{M}_{3,2})$. We now simplify this expression. We first expand this sum, leading to the equality
\begin{eqnarray*}
 c_{3}(\mathcal{F}-E_2) &=& 735K_{2}^{3} + 175(K_{1} - \lambda_{1})K_{2}^{2}   + 21 (K_{1}^{2} -  \lambda_{1} K_{1} + \lambda_{1}^{2}- \lambda_{2}) K_{2}\\
  & &  + 1400 \Delta_{12} K_{2}^{2}  + 210 \Delta_{12} K_{1} K_{2}   - 210 \Delta_{12} \lambda_{1} K_{2} + 840 \Delta_{12}^{2}  K_{2}\\
  & &  + K_{1}^{3}  - \lambda_{1} K_{1}^{2}  + 12 \Delta_{12} K_{1}^{2}   + (\lambda_{1}^{2} - \lambda_{2}) K_{1} - 12 \Delta_{12} \lambda_{1} K_{1} + 60 \Delta_{12}^{2}  K_{1} - \lambda_{3} \\                                 
  & &  + 2 \lambda_{1} \lambda_{2} - 12 \Delta_{12} \lambda_{2} - \lambda_{1}^{3} + 12 \Delta_{12} \lambda_{1}^{2}  - 60 \Delta_{12}^{2}  \lambda_{1} + 160 \Delta_{12}^{3}\,.
\end{eqnarray*}

We now simplify this expression using the equalities recalled in the first part of Lemma~\ref{lem:formularium}. For example, the second equality of this lemma implies that the coefficient $60 \Delta_{12}^{2}  K_{1}$ is equal to $-60\Delta_{12}K_1^{2}$.
Doing this for all the  coefficient, this gives the expression 
\begin{eqnarray*}
 c_{3}(\mathcal{F}-E_2) &=& 735K_{2}^{3} + 175(K_{1} - \lambda_{1})K_{2}^{2}+ 21 (K_{1}^{2} -  \lambda_{1} K_{1} + \lambda_{1}^{2}- \lambda_{2})  K_{2} \\
 & & + (882K_{1}^{2}-162\lambda_{1}K_{1}+12(\lambda_{1}^{2}-\lambda_{2}))\Delta_{12} \\
 & & + K_{1}^{3}-\lambda_{1}K_{1}^{2}+(\lambda_{1}^{2}-\lambda_{2})K_{1} - \lambda_{3}+ 2 \lambda_{1} \lambda_{2}-\lambda_{1}^{3}  \,.
\end{eqnarray*}
\par
We now forget the two marked points to obtain the $\lambda$-class of $\mathcal{M}_{3}(6;-2)$ in the Picard group of $\mathcal{M}_{3}$.
By forgetting the second point and using the formulas of the second part of Lemma~\ref{lem:formularium} we obtain
\begin{eqnarray*}
 \pi_{2,\ast} \left( c_{3}(\mathcal{F}-E(2)) \right)  &=& 735\kappa_{2} + 175(K_{1} - \lambda_{1})\kappa_{1}+ 21 (K_{1}^{2} -  \lambda_{1} K_{1} + \lambda_{1}^{2}- \lambda_{2})  \kappa_{0} \\
  & & + 882K_{1}^{2}-162\lambda_{1}K_{1}+12\lambda_{1}^{2}-12\lambda_{2} \,.
\end{eqnarray*}
And finally by forgetting the first point we get
\begin{equation*}
 \pi_{1,\ast} \circ \pi_{2,\ast} \left( c_{3}(\mathcal{F}-E(2)) \right) =  882\kappa_{1}-162\kappa_{0}\lambda_{1} + 175\kappa_{0}\kappa_{1} -21\kappa_{0}^{2}\lambda_{1} +21\kappa_{0}\kappa_{1}\,.
\end{equation*}
Hence using that $\kappa_{0}=4$ and $\kappa_{1}= 12\lambda_{1}$ we obtain that the class of $  \pi_{1,\ast} \circ \pi_{2,\ast} \left( c_{3}(\mathcal{F}-E(2)) \right)  $ in $\Pic(\moduli[3])$ is
\begin{equation*}
 \left[  \pi_{1,\ast} \circ \pi_{2,\ast} \left( c_{3}(\mathcal{F}-E(2)) \right) \right] =  19008 \lambda_{1} \,.
\end{equation*}
\smallskip
\par

The degeneracy locus of the map $\phi \colon E_{2} \to \calF$ contains the locus that we are interested in of curves $X$ together with points $w$ and $z$ such that there exists a differential with divisor $6z - 2w$ and a locus supported in the diagonal~$\Delta$. So in order to obtain the class of $\moduli[3](6;-2)$ in the Picard group of $\moduli[3]$ it remains to substract this contribution. This is in general a delicate task and we refer to \cite{harKod2,cukfamwei,esthyp} for some examples.
\par
At a point $(X,z,z)$ of the diagonal, the evaluation map restricts to 
$$  \phi_z\colon H^0(X,K_X(2z))\to H^0(X,K_X(2z)/K_X(-4z)) \,.$$
This has generically rank $4$ as expected and has rank $3$ exactly if there is an abelian differential on $X$ which has a zero of order $4$ at~$z$.  The class of this locus is classical but we recall it for completeness.
We have
\[  c(\calF) = 1+ 10 K_{1} +35 K_{1}^{2}  + \cdots\,. \]
Hence 
\[  c_{2}(\calF - E) = \lambda_{1}^{2} - 10 \lambda_{1} K_{1} +35 K_{1}^{2} \,. \]
When pushing down to $\moduli[3]$ we obtain
\begin{eqnarray*}
 \pi_{1,\ast}( c_{2}(\calF - E)) &=& 35 \kappa_{1} -10 \lambda_{1} \kappa_{0} \\
 &=& 35 \cdot 12\lambda_{1} - 10 \cdot 4 \lambda_{1}\\
 &=& 380 \lambda_{1}\,,
\end{eqnarray*}
where we used again that $\kappa_{1} = 12 \lambda_{1}$ and $\kappa_{0} = 4$.

Now we compute the multiplicity of this locus. Consider a point $(X,z)$ where $z$ is such that there exists a differential $\omega$ on $X$ whose divisor  is~$4z$. Locally the vector bundle $E_{2}$ is generated by four sections. Let $u$ is a local coordinate of $X$ at~$z$, we can take these sections to be $\eta_{t}=\tfrac{du}{u^{2}}$, the family of  differentials $u^{2}\eta_{t}$, a family of differentials vanishing at order $1$ or $2$ at $z$ depending if $X$ is non hyperelliptic or hyperelliptic respectively and finally a family of the form $\omega_{t} := u^{2}(u^{4}+s)\omega_{0}$ where $s$ is a function of the base vanishing at~$X$. So the locus of pointed curves $(X,z)$ is given by the scheme given by the two by two  minors of the matrix
\[ \left(
  \begin{array}{ c c c c}
     1 & u^{2} & u^{3}h & u^{2}(u^{4}+s) \\
     0 & 2u & \partial_{u}( u^{3}h) &6u^{5}+2u s
  \end{array} \right)\,,
\]
where $h$ depends if the point is \Weierstrass or not. It is easy to check that the column where it appears the function $h$ will  not contribute to the multiplicity. In fact, the multiplicity of this point is given by the order of vanishing of~$s$. In order to compute this order note that $\omega_{t}$ leads to a family of of multi-scale differentials after blowing up the diagonal~$\Delta$. Note that the equation of the node is then given by $xy=f$ where $f$ is the function defining the diagonal. Moreover, the function $s$ is a rescaling parameter for this family of multi-scale differentials as defined in \cite[Section~7.1]{complis}. More precisely, according to Equation~(7.1) and Definition~7.2 of \cite{complis}, we have $s=f^{5}$, where the $5$ is the prong number. This implies that  the multiplicity of the considered locus is~$5$.

Summing up all the computations of this section, we obtain that the class of  $\moduli[3](6;-2)$ in the Picard group of $\moduli[3]$ is 
\begin{equation}\label{eq:lambdaclasse}
 \left[ \moduli[3](6;-2) \right] =  19008 \lambda_{1} - 5 \cdot 380 \lambda_{1}  = 17108\lambda_{1} \,.
\end{equation}

\section{Class of the divisor $\barmoduli[3](6;-2)$}
\label{sec:class}

In this section we compute the class of the locus $\barmoduli[3](6;-2)$  in the rational Picard group $\Pic(\barmoduli[3])\otimes\bQ$, proving Theorem~\ref{thm:classeismenosdos}. We first give in Section~\ref{sec:gendegenfam} some general facts on degenerating families of curves and differentials. Then we solve in Section~\ref{sec:enum} some enumerative problems related to differentials of the second kind. And finally in Section~\ref{sec:compu} we use this information to compute the class $[\moduli[3](6;-2)]$ using the method of test curves.

\subsection{General facts on degenerating families}
\label{sec:gendegenfam}

First we consider the problem of finding out how many fibres are isomorphic in a family of semi-stable curves. This generalizes a result of \cite[Section 1]{Cornalba} that we could not find in the literature in the following general form.

\begin{lem}\label{lem:gradosemiastable}
  Let $f\colon \calX \to (B,p)$ be a family of generically smooth semi-stable curves over a smooth curve such that the fibre $X_{p}$ has exactly one unstable rational component~$E$. We denote by $n_{i}$ the two nodes of~$E$.  Given   a local parameter $t$ of $B$ at $p$, if the family is of the form $x_{1}y_{1}=t^{a_{1}}$ and $x_{2}y_{2}=ct^{a_{2}}$ with $|c| =1$  at the nodes $n_{i}$,  then the induced application $\mu_{f}\colon B \to \barmoduli$ is a  cover ramified at $p$ of degree  $a_{1}+a_{2}$ on its image. 
 \end{lem}

 \begin{proof}
 The proof has two parts.   First we are going to give a condition so that two curves are isomorphic near $X_{p}$. Then we are going to count the number of curves that satisfy that criterion in the hypotheses of the lemma. 
 \smallskip
 \par
 Let $X_{p}$ be a semi-stable curve as in Lemma~\ref{lem:gradosemiastable} with  exceptional component $E$. Consider the family $\calY \to \Delta^{2} = \left\{ (t_{1},t_{2}) : |t_{i}| < 1 \right\}$ which is given locally at the nodes by the equations $x_{i}y_{i}=t_{i}$ where $x_{i}$ are local equations of $E$ in $\calY$ and $y_{1}y_{2}=1$. Consider the two curves $X_{\theta_{1},\theta_{2}}$ and $X_{\vartheta_{1},\vartheta_{2}}$ above the parameters $ (r_{1}\exp(i\theta_{1}),r_{2}\exp(i\theta_{2}))$ and $(r_{1}\exp(i\vartheta_{1}),r_{2}\exp(i\vartheta_{2}))$. We now show that these curves are isomorphic if and only if $\theta_{1}+\theta_{2}=\vartheta_{1}+\vartheta_{2}$. 
 \par
Let us study when the identity 
 of a neighborhood of a complement of $x_{i}=0$ in the curves $X_{\theta_{1},\theta_{2}}$ and $X_{\vartheta_{1},\vartheta_{2}}$ can be extend to an isomorphism $\varphi\colon X_{\theta_{1},\theta_{2}} \to X_{\vartheta_{1},\vartheta_{2}}$. The identity can be extended through the annuli $x_{i}y_{i}=t_{i}$ by the functions 
 $$\phi_{i}\colon (x_{i},y_{i}) \mapsto (x_{i},\exp(i(\theta_{i}-\vartheta_{i}))y_{i})\,.$$
The functions $\phi_{i}$ can be extended to an holomorphic fonction on $E$ if and only if $\phi_{1}$ coincides with $\phi_{2}$ on~$E$, that is 
$$\exp(i(\theta_{1}-\vartheta_{1}))y_{1} = \exp(-i(\theta_{2}-\vartheta_{2}))y_{2} = \exp(-i(\theta_{2}-\vartheta_{2}))y_{1}^{-1} \,. $$
  From that equation it follows directly that the curves $X_{\theta_{1},\theta_{2}}$ y $X_{\vartheta_{1},\vartheta_{2}}$ are isomorphic if  and only if $\theta_{1}-\vartheta_{1}+\theta_{2}-\vartheta_{2}=0$.
 \smallskip
 \par
 Now we fix two integers $a_{1},a_{2}\geq1$ as in lemma~\ref{lem:gradosemiastable} and a real number $r>0$. We introduce the curve 
 $$\calC_{a_{1},a_{2},c} = \left\{ (a_{1}\theta,c+a_{2}\theta) : \theta \in \mathbb{S}^{1}\right\} \subset \mathbb{S}^{1}\times \mathbb{S}^{1} \, .$$
It follows from the first part of the proof that it suffices to show that  the antidiagonal $\varDelta = \left\{ (\theta,-\theta) : \theta \in \mathbb{S}^{1}\right\}$ in $\mathbb{S}^{1}\times \mathbb{S}^{1}$  and the locus $\calC_{a_{1},a_{2},c}$ have $a_{1}+a_{2}$ points of intersection. Since for every $c$ the curve $\calC_{a_{1},a_{2},c}$ is a translation of $\calC_{a_{1},a_{2},0}$ the number of intersections does not depend on~$c$.  Assume first that $a_{1}$ and $a_{2}$ are prime to each other. The intersections between $\varDelta$ and $\calC_{a_{1},a_{2},0}$  are given by the points $(k/(a_{1}+a_{2}),-k/(a_{1}+a_{2}))$ with $k\in \left\{0,\dots,a_{1}+a_{2}-1\right\}$. Hence there are $a_{1}+a_{2}$ points of intersection between  the loci $\varDelta$ and $\calC_{a_{1},a_{2},0}$.
 If  $d=\gcd(a_{1},a_{2})>1$ then the curves $\varDelta$ and $\calC_{a_{1},a_{2},0}$ have  $a_{1}/d + a_{2}/d$ points of intersection, each of multiplicity~$d$.
 \end{proof}

 The following result will be very useful in order to compute some intersection numbers in Section~\ref{sec:compu}. 
\begin{lem}\label{lem:passurnoeud}
 Given a twisted differential $(X;z,w;\omega)$ of type $(a;-b)$ of  genus $g\geq 2$. The pointed curve~$(X;z,w)$ is not one of the curves pictured in Figure~\ref{fig:nopuente} where~$E$ is a rational curve.  
  \begin{figure}[ht]
 \centering
\begin{tikzpicture}[scale=1.2]
\begin{scope}[xshift=-4cm]
 \draw (-4,0.1) coordinate (x1) .. controls (-3.4,-.3) and (-3.6,-.6) .. (-3,-1.1) coordinate (q1) ;
\node[above] at (x1) {$X_{1}$};
\draw (-3.1,-1.1) coordinate (x1) .. controls (-2.6,-.9) and (-2.4,-.9) .. (-1.9,-1.1) coordinate (q1) coordinate [pos=.5] (z1);
\fill (z1) circle (1pt); \node[above] at (z1) {$z$};
\draw (-2.1,-1.1) coordinate (x1) .. controls (-1.5,-.6) and (-1.7,-.3) .. (-1.1,.1) coordinate (q2) coordinate [pos=.5] (z1);
\node[above] at (q2) {$X_{2}$};
\fill (z1) circle (1pt);\node[right] at (z1) {$w$};
\node[right] at (q1) {$E$};
\end{scope}

\draw (-4,0.1) coordinate (x1) .. controls (-3.4,-.3) and (-3.6,-.6) .. (-3,-1.1) coordinate (q1) ;
\node[above] at (x1) {$X_{1}$};
\draw (-3.1,-1.1) coordinate (x1) .. controls (-2.6,-.9) and (-2.4,-.9) .. (-1.9,-1.1) coordinate (q1) coordinate [pos=.5] (z1);
\fill (z1) circle (1pt); \node[above] at (z1) {$w$};
\draw (-2.1,-1.1) coordinate (x1) .. controls (-1.5,-.6) and (-1.7,-.3) .. (-1.1,.1) coordinate (q2) coordinate [pos=.5] (z1);
\node[above] at (q2) {$X_{2}$};
\fill (z1) circle (1pt);\node[right] at (z1) {$z$};
\node[right] at (q1) {$E$};

\begin{scope}[xshift=4cm]
 \draw (-4,0.1) coordinate (x1) .. controls (-3.4,-.3) and (-3.6,-.6) .. (-3,-1.1) coordinate (q1) ;
\node[above] at (x1) {$X_{1}$};
\draw (-3.1,-1.1) coordinate (x1) .. controls (-2.6,-.9) and (-2.4,-.9) .. (-1.9,-1.1) coordinate (q1) coordinate [pos=.3] (z1)coordinate [pos=.7] (z2);
\fill (z1) circle (1pt); \node[above] at (z1) {$z$};
\fill (z2) circle (1pt);\node[above] at (z2) {$w$};
\draw (-2.1,-1.1) coordinate (x1) .. controls (-1.5,-.6) and (-1.7,-.3) .. (-1.1,.1) coordinate (q2) ;
\node[above] at (q2) {$X_{2}$};
\node[right] at (q1) {$E$};

\end{scope}
\end{tikzpicture}
\caption{The stable pointed curve $(X;z,w)$ with one projective bridge.}
\label{fig:nopuente}
\end{figure}
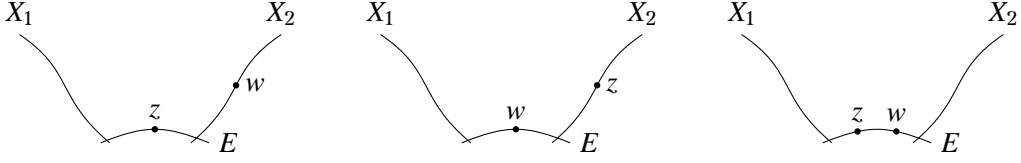
\end{lem}

Before doing the proof, we want to remark that it seems to be a general principle that special points related to differentials do not easily converge to a separating node. On the other hand, we will see in Section~\ref{sec:enum} that special points can easily converge to separating nodes. For an other instance of this principles, the reader can look at  \cite{weiernode}.

\begin{proof}
In these three cases the restriction of $\omega$ to $E$ has at least two poles. Hence Theorem~1.5 of \cite{geta} implies that the residues of these nodes are different from zero. Hence the global residue condition, that is recalled in Definition~\ref{def:twisteddif}, can not be satisfied on any of these curves.
\end{proof}

Finally, we give the form of the curve underlying a multi-scale differential of second kind with a unique zero in genus~$1$.

\begin{lem}\label{lem:degenposgun}
 The  multi-scale differential $(X;z,w_{1},\dots,w_{p},\omega;\sigma;\cleq)$ in the boundary of the locus $\ores[1](a;-b_{1},\dots,-b_{p})$ of differentials of the second kind are such that $X$ is either irreducible (with one node) or is two rational curves which are glued together at two nodes.
\end{lem}

\begin{proof}
If $X$ was a curve not of one type given in Lemma~\ref{lem:degenposgun}, then either $X$ has a rational component with marked points glued at the rest of the curve at one node or it is a chain of rational components  of length strictly greater than~$2$.
\par
In the first case the differential on the projective curve has exactly one zero. Hence by \cite{geta} it has some poles with non zero residues. This implies that this differential can not be in the limit of a family of differential of $\ores[1](a;-b_{1},\dots,-b_{p})$ as proved in \cite[Theorem 1.1]{muHur}.
\par 
In the second case, consider the component $X_{1}$ of $X$ which contains the zero $z$ of the multi-scale differential. Then the component $X_{1}$ is the unique local minimum for the full order $\cleq$ as shown in \cite[Lemma 3.9]{IVC}. This implies that the differential~$\omega|_{X_{1}}$ has two poles at the nodal points of $X_{1}$. On the adjacent components $X_{2}$ and $X_{3}$ the differentials~$\omega|_{X_{i}}$ have a zero at the corresponding nodal point. Hence at the other nodal points the differentials on $\omega|_{X_{2}}$ and $\omega|_{X_{3}}$ have another zero. This implies that there are both higher for $\cleq$ than some other components of~$X$. Hence there is another local minimum, contradicting the unicity of the local minimum for~$\cleq$.
\end{proof}

\subsection{Some enumerative problems}
\label{sec:enum}

Let us begin with a proposition giving the number of differentials with a unique zero and a unique pole on a curve of genus $2$. This is essentially known (see  \cite[Proposition 2.2]{chentara} or \cite[Section 2.6]{mudiv}) but we recall since it is only implicitly stated in the cited articles.
\begin{prop}\label{prop:gradg2}
For every $a\geq 2$, the application $\pi\colon\PP\omoduli[2](a+2;-a)\to\moduli[2]$ is an  unramified cover of degree $2(a+2)^{2}a^{2}-18$.
\end{prop}

\begin{proof}
 Let $C$ be a general curve of genus $2$ and consider the map $$f\colon C^{2} \to \Pic^{2}(C) : (p_{1},p_{2}) \mapsto (a+2)p_{1} -a p_{2} \,.$$ By \cite{chentara} and \cite{mudiv} we know that~$f$ is a finite map of degree  $2(a+2)^{2}a^{2}$. Moreover, this application is ramified at the loci
 $$\Delta = \left( (p_{1},p_{2})\in C^{2} : p_{1}=p_{2} \right)$$
and 
 $$ \calK = \left( (p_{1},p_{2})\in C^{2} : p_{1}=\iota(p_{2}) \right)\,,$$
 where $\iota$ is the hyperelliptic involution. Note that $\Delta \cap \calK$ is the set of the six \Weierstrass points of $C$.  Let us consider the preimage by $f$ of the canonical bundle of~$C$. If a pair $(p_{1},p_{2})$ is in the preimage of $K_{C}$ belongs to $\Delta$, then $(a+2-a)p_{1}$ is canonical, so $p_{1}$ is a \Weierstrass point. Second, if a pair $(p_{1},p_{2})$  of the preimage of $K_{C}$ belongs to $\calK$, then there are differentials $\omega_{1}$ and $\omega_{2}$ whose respective divisors are $(a+2)p_{1}-ap_{2}$ and $p_{1}+p_{2}$. Then $\eta = \omega_{1}\omega_{2}^{a}$ is a $(a+1)$-differential with a unique zero of order $2a+2$. The $(a+1)$-differential $\eta$ exists on a general curve if and only if $\eta$ is the $(a+1)$th power of an abelian differential with a double zero. This implies that $p_{1}=p_{2}$ is a \Weierstrass point. 
 \par
 We now compute the degree of ramification of $f$ at the \Weierstrass points as in \cite{mudiv}. Let $p$ be a \Weierstrass point and $(\omega_{1},\omega_{2})$ generators of $H^{0}(C,K_{C})$ such that $\ord_{p}(\omega_{1})=0$ and $\ord_{p}(\omega_{2})=2$. The derivative of $f$ at $p$ is of the form
 \[Df(p)=\begin{pmatrix}
   1 & 1 \\
   t_{1}^{2} & t_{2}^{2} 
\end{pmatrix}\]
where $t_{i}$ are functions which vanish at $p$ at order $1$. Hence the determinant of this matrix vanishes at order $2$ at $p$. It follows that the order of branching at $p$ is equal to $2$. 
Hence there are $18$ pairs in the preimage of the canonical divisor $K_{C}$ that do not correspond to solutions in the stratum $\omoduli[2](a+2;-a)$ but to abelian differentials with a single zero order~$2$. 
\par
It remains to show that the cover is not branched. Suppose it is branched, then the branching point corresponds to a twisted differential where the underlying curve is a genus $2$ curve with a rational curve $E$ glued at a point. The curve $E$ contains the zero of order $a+2$ and the pole of order $-a$. Hence the restriction of the differential to $E$ has two poles and a unique zero, hence the residues at the poles are non-zero.  Hence this twisted differential does not satisfy the global residual condition of Definition~\ref{def:twisteddif} and is not smoothable.
\end{proof}
\par
We will now deduce an important consequence of this result for the locus of differentials of second kind on elliptic curves. 

\begin{cor}\label{cor:res0g1prim}
 The degree of the map $\pi\colon\PP\ores[1](a+2;-a,-2)\to\moduli[1,1]$ forgetting the pole is equal to $(a+2)^{2} + a^{2} -10$ for $a\geq3$ and $5$ for $a=2$.
\end{cor}

\begin{proof}
Let us consider a stable curve $X$ of genus $2$ which is the union of two elliptic curves attached at one node. We describe all the multi-scale differentials on a curve $X$ semi-stably equivalent to $X$. By Lemma~\ref{lem:passurnoeud} the pointed curves which can appear in such multi-scale differentials  are of the form pictured in Figure~\ref{fig:posslimites}. 
\par
  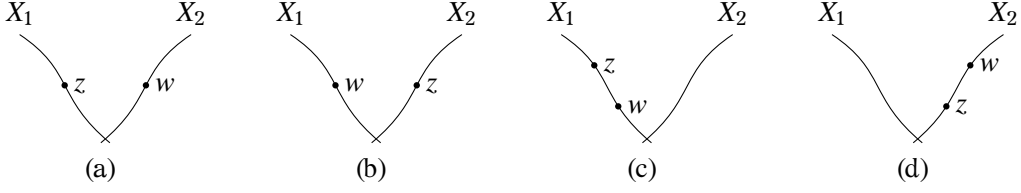
\begin{figure}[ht]
 \centering
\begin{tikzpicture}[scale=1.2]
\begin{scope}[xshift=-3cm]
\draw (-4,0.1) coordinate (x1) .. controls (-3.4,-.3) and (-3.6,-.6) .. (-3,-1.1) coordinate (q1) coordinate [pos=.5] (z1);
\fill (z1) circle (1pt); \node[right] at (z1) {$z$};
\node[above] at (x1) {$X_{1}$};
\draw (-3.1,-1.1) coordinate (x1) .. controls (-2.5,-.6) and (-2.7,-.3) .. (-2.1,.1) coordinate (q2) coordinate [pos=.5] (z1);
\node[above] at (q2) {$X_{2}$};
\fill (z1) circle (1pt);\node[right] at (z1) {$w$};

\node at (-3.1,-1.4) {(a)};
\end{scope}

\draw (-4,0.1) coordinate (x1) .. controls (-3.4,-.3) and (-3.6,-.6) .. (-3,-1.1) coordinate (q1) coordinate [pos=.5] (z1);
\fill (z1) circle (1pt); \node[right] at (z1) {$w$};
\node[above] at (x1) {$X_{1}$};
\draw (-3.1,-1.1) coordinate (x1) .. controls (-2.5,-.6) and (-2.7,-.3) .. (-2.1,.1) coordinate (q2) coordinate [pos=.5] (z1);
\node[above] at (q2) {$X_{2}$};
\fill (z1) circle (1pt);\node[right] at (z1) {$z$};

\node at (-3.1,-1.4) {(b)};

\begin{scope}[xshift=3cm]
\draw (-4,0.1) coordinate (x1) .. controls (-3.4,-.3) and (-3.6,-.6) .. (-3,-1.1) coordinate (q1) coordinate [pos=.3] (z1)coordinate [pos=.7] (z2);
\fill (z1) circle (1pt); \node[right] at (z1) {$z$};
\node[above] at (x1) {$X_{1}$};
\draw (-3.1,-1.1) coordinate (x1) .. controls (-2.5,-.6) and (-2.7,-.3) .. (-2.1,.1) coordinate (q2) ;
\node[above] at (q2) {$X_{2}$};
\fill (z2) circle (1pt);\node[right] at (z2) {$w$};

\node at (-3.1,-1.4) {(c)};
\end{scope}

\begin{scope}[xshift=6cm]
\draw (-4,0.1) coordinate (x1) .. controls (-3.4,-.3) and (-3.6,-.6) .. (-3,-1.1) coordinate (q1) ;
\node[above] at (x1) {$X_{1}$};
\draw (-3.1,-1.1) coordinate (x1) .. controls (-2.5,-.6) and (-2.7,-.3) .. (-2.1,.1) coordinate (q2) coordinate [pos=.3] (z1)coordinate [pos=.7] (z2);
\fill (z1) circle (1pt); \node[right] at (z1) {$z$};
\node[above] at (q2) {$X_{2}$};
\fill (z2) circle (1pt);\node[right] at (z2) {$w$};

\node at (-3.1,-1.4) {(d)};
\end{scope}
\end{tikzpicture}
\caption{The pointed curves $(X;z,w)$ which can appear in multi-scale differentials at the boundary of the stratum $\omoduli[2](a+2;-a)$.}
\label{fig:posslimites}
\end{figure}
\par
In case (a) the differential $\omega_{1}$ on $X_{1}$ is in the stratum $\omoduli[1](a+2;-a-2)$ and the differential $\omega_{2}$ on~$X_{2}$ is in the stratum $\omoduli[1](a;-a)$. Hence there are $(a+2)^2-1$ different differentials on $X_{1}$ and $a^2-1$ different differentials on $X_{2}$. Since there is clearly only one class of prong-matching there are $\left((a+2)^2-1\right)\left(a^2-1\right)$ different multi-scale differentials of this type. Moreover, the space of multi-scale differentials is smooth at this point, hence  $\left((a+2)^2-1\right)\left(a^2-1\right)$ degenerate to this type. The type (b) is similar. Hence the intersection number with the ones of type (c) (or (d)) is equal to
\begin{eqnarray*}
D &=&  ((a+2)^{2}a^{2} -9) - \left( ((a+2)^{2}-1)(a^{2} -1) \right) \,,\\
  &=& 2a^2+4a -6\,,\\
  &=& (a+2)^2+a^2-10\,.
\end{eqnarray*}
Since the space of multi-scale differentials is smooth at these points and there is a unique differential on $X_{2}$ (in the case (c)), this implies that the number of distinct differentials of the second kind in $\omoduli[1](a+2;-a,-2)$ is $(a+2)^{2} + a^{2} -10$ for $a\geq3$. The case $a=2$ is similar taking into account that there is a symmetry since both poles have the same order.
\end{proof}
\par
We now compute the degree of the map $\pi\colon \PP\ores[1](6;-2,-2,-2) \to \moduli[1,1]$ proving Theorem~\ref{thm:degreeprojs}. In order to illustrate the ideas of the proof in easier context, we first compute the degrees of the maps $\pi\colon \PP\ores[1](4;-2,-2) \to \moduli[1,1]$ and $\pi\colon \PP\ores[1](5;-3,-2) \to \moduli[1,1]$ (that is already known by Corollary~\ref{cor:res0g1prim}).
\begin{lem} \label{lem:gradoexs}
 The degree of the map $\pi\colon \PP\ores[1](4;-2,-2) \to \moduli[1,1]$ forgetting the poles  is~$5$ and the degree of the application $\pi\colon \PP\ores[1](5;-3,-2) \to \moduli[1,1]$ is $42$. 
\end{lem}

The proof of this result is by degeneration in the moduli space of multi-scale differentials. The results that we use here are recalled in Section~\ref{sec:msd}.

\begin{proof}
 We first consider the case of the locus $\ores[1](4;-2,-2)$ of differentials of the second kind of type $(4;-2,-2)$.  By Lemma~\ref{lem:degenposgun} there are $3$ types of pointed curves on which there can exists a multi-scale differential in the closure of this locus. This is summarised in Table~\ref{tab:4-2-2} that we first explain in details. Note that Tables~\ref{tab:5-3-2} and~\ref{tab:6-2-2-2} have the same entries.
 \par
 \begin{table}[h]
\begin{tabular}{|c|c|c|c|}
  \hline
  Pointed curve & 
  \begin{tikzpicture}
\draw [] (-.5, 0) coordinate (x1)
  .. controls ++(0:1) and ++(0: .3) .. ( 0, -1) coordinate [pos=.5] (w2)  coordinate  (z)
  .. controls ++(180:.3) and ++(180:1) .. ( .5, 0) coordinate [pos=.5] (w1);
  \fill (z) circle (1pt);\node[below] at (z) {$z$};
  \fill (w1) circle (1pt);\node[left] at (w1) {$w_{1}$};
\fill (w2) circle (1pt);\node[right] at (w2) {$w_{2}$};

%
\end{tikzpicture}
& 
  \begin{tikzpicture}
\draw (-.5,-.2)  coordinate (q1) --  (.5,-.2) coordinate (q2)coordinate [pos=.5] (w2);
\draw (-.5,0) .. controls (-.3,-1.3) and (.3,-1.3) .. (.5,0) coordinate [pos=.25] (w1) coordinate [pos=.5] (z);

\node [above left] at (q1) {$\kappa_{1}$};\node [above right] at (q2) {$\kappa_{2}$};

\fill (z) circle (1pt); \node [below] at (z) {$z$};
\fill (w1) circle (1pt);\node [left] at (w1) {$w_{1}$};
\fill (w2) circle (1pt);\node [above] at (w2) {$w_{2}$};
\end{tikzpicture}
&
  \begin{tikzpicture}
 \draw (-.8,-.2)  coordinate (q1) --  (.8,-.2) coordinate (q2)coordinate [pos=.35] (w1)coordinate [pos=.65] (w2);
\draw (-.8,0) .. controls (-.3,-1.3) and (.3,-1.3) .. (.8,0)  coordinate [pos=.5] (z);

\node [above left] at (q1) {$\kappa_{1}$};\node [above right] at (q2) {$\kappa_{2}$};

\fill (z) circle (1pt); \node [below] at (z) {$z$};
\fill (w1) circle (1pt);\node [above] at (w1) {$w_{1}$};
\fill (w2) circle (1pt);\node [above] at (w2) {$w_{2}$};

\end{tikzpicture}
\\
  \hline
  \#  twisted differentials & 1 & 1 & 1 \\
  \hline
  \#   prong-matching & 1 & 1 & 2 \\
  \hline
 Local degree & 1 & 2 & 2 \\
  \hline
   Symmetries & 1 & 1 & 2 \\
  \hline
  Total count & 1 & 2 & 2  \\
  \hline
\end{tabular}
\caption{Summary of the case $\ores[1](4;-2,-2)$.}
\label{tab:4-2-2}
\end{table}
 \par 
 The first line gives the pointed stable curves on which there may exist a multi-scale differential. This is only given up to permuting the poles of the same order, for example in the second column, the case where $w_{1}$ and $w_{2}$ are permuted gives the same differential. The numbers $\kappa_{i}$ are the prong numbers at the corresponding node.  
 \par
 The second line gives the number of twisted differentials that exit on the pointed curve under consideration. 
 \par
 The third line gives the number of classes of prong-matching. In the cases we will consider here, the multi-scale differentials have only two nodes and two levels. Hence this number is simply given by $\gcd(\kappa_{1},\kappa_{2})$.
 \par
 The fourth line gives the local degree of  the stabilisation map $\moduli[1,m] \to \moduli[1,1]$ restricted to the smooth pointed curves underlying the differentials obtained by smoothing the multi-scale differential under consideration. This degree is given by Lemma~\ref{lem:gradosemiastable}. Note that if there is more than one class of prong-matching, the degree is given for each choice of prong-matching but it does not depend on this choice.
 \par
 According to the fourth line, on a smooth curve near the boundary of $\moduli[1,1]$ the number of pointed differentials on this curves which degererate to the multi-scale differential under consideration is equal to the degree. However, it happens that these pointed differentials only differ by a permutation of the marked points.
 The fifth line gives the order of the group generated by these permutations of the marked points. 
 \par
 Finally in the last line we group all these informations to give the number of differentials of  type $\mu$ on a smooth curve which degenerate to a multi-scale differential such that the underlying pointed curve is the one of the first line.
 \smallskip
 \par
 We now come back to the case of $\ores[1](4;-2,-2)$.
 Consider the irreducible pointed curve pictured in the first column. Any multi-scale differential on this pointed curve has two single poles at the node. Hence the pull-back of the differential on the normalisation of the curve  lies in the stratum  $\omoduli[0] (4;-2,-2,-1,-1)$ and the residues at the poles of order $2$ are equal to~$0$. There is a unique such differential which is shown in Figure~\ref{fig:diffgcerocasoa}. 
\begin{figure}[htb]
 \center
\begin{tikzpicture}
\begin{scope}[xshift=-.4cm,yshift=-.8cm]
\coordinate (a) at (-1,1);
\coordinate (b) at (0,1);

    \fill[fill=black!10] (a)  -- (b)coordinate[pos=.5](f) -- ++(0,1) --++(-1,0) -- cycle;
    \fill (a)  circle (2pt);
\fill[] (b) circle (2pt);
 \draw  (a) -- (b);
 \draw (a) -- ++(0,.9) coordinate (d)coordinate[pos=.5](h);
 \draw (b) -- ++(0,.9) coordinate (e)coordinate[pos=.5](i);
 \draw[dotted] (d) -- ++(0,.2);
 \draw[dotted] (e) -- ++(0,.2);
\node[above] at (f) {$1$};
\end{scope}

\begin{scope}[xshift=-.4cm,yshift=-1.2cm]
\coordinate (a) at (-1,1);
\coordinate (b) at (0,1);

\fill[fill=black!10] (a)  -- (b)coordinate[pos=.5](f) -- ++(0,-1) --++(-1,0) -- cycle;
\fill (a)  circle (2pt);
\fill[] (b) circle (2pt);
\draw  (a) -- (b);
\draw (a) -- ++(0,-.9) coordinate (d)coordinate[pos=.5](h);
\draw (b) -- ++(0,-.9) coordinate (e)coordinate[pos=.5](i);
\draw[dotted] (d) -- ++(0,-.2);
\draw[dotted] (e) -- ++(0,-.2);
\node[below] at (f) {$3$};
\end{scope}

\begin{scope}[xshift=1cm,yshift=0cm]
\fill[fill=black!10] (0,0) coordinate (Q) circle (1.1cm);
\coordinate (a) at (-.5,0);
\coordinate (b) at (.5,0);
\fill (a)  circle (2pt);
\fill (b) circle (2pt);
\draw (a) -- (b)coordinate[pos=.5](d);

\node[below] at (d) {$2$};
\node[above] at (d) {$3$};
\end{scope}

\begin{scope}[xshift=3.4cm,yshift=0cm]
\fill[fill=black!10] (0,0) coordinate (Q) circle (1.1cm);
\coordinate (a) at (-.5,0);
\coordinate (b) at (.5,0);
\fill (a)  circle (2pt);
\fill (b) circle (2pt);
\draw (a) -- (b)coordinate[pos=.5](d);

\node[below] at (d) {$1$};
\node[above] at (d) {$2$};
\end{scope}
\end{tikzpicture}
\caption{The differential of $\omoduli[1](4;-2,-2,-1,-1)$ such that the residues are $(0,0,1,-1)$.} \label{fig:diffgcerocasoa}
\end{figure}
Moreover, the number of prong matching and the local degree are clearly equal to $1$. Hence on a smooth curve near the boundary of $\moduli[1,1]$ there is a unique differential which degenerate to a multi-scale differential with this underlying pointed curve.
 \par
 We now consider the curve pictured in the second column of Table~\ref{tab:4-2-2}. First note that the prong numbers have to be equal to~$1$, since otherwise there would exist no full order~$\cleq$. Hence the differential on the upper component is in the stratum $\omoduli[0](0,0;-2)$ and the differential on the lower component is in the stratum  $\omoduli[0](4;-2,-2,-2)$ and the residue vanishes at the marked pole. There is clearly a unique differential in $\omoduli[0](0,0;-2)$. Moreover \cite{chenchen2} show that there is a unique differential satisfying the second conditions which is shown in Figure~\ref{fig:diffgcerocasob}.  
 \begin{figure}[htb]
 \center
\begin{tikzpicture}
\begin{scope}[xshift=0cm]
\fill[fill=black!10] (0,0)coordinate (Q)  circle (1.1cm);
    \coordinate (a) at (-.5,0);
    \coordinate (b) at (.5,0);

     \fill (a)  circle (2pt);
\fill[] (b) circle (2pt);
    \fill[white] (a) -- (b) -- ++(0,-1.1) --++(-1,0) -- cycle;
 \draw  (a) -- (b);
 \draw (a) -- ++(0,-1);
 \draw (b) -- ++(0,-1);

\node[above] at (Q) {$1$};
    \end{scope}

\begin{scope}[xshift=2.5cm]
\fill[fill=black!10] (0,0) coordinate (Q) circle (1.1cm);
\coordinate (a) at (-.5,0);
\coordinate (b) at (.5,0);
\fill (a)  circle (2pt);
\fill (b) circle (2pt);
\draw (a) -- (b)coordinate[pos=.5](d);

\node[below] at (d) {$1$};
\node[above] at (d) {$2$};
\end{scope}

\begin{scope}[xshift=4.5cm]
\fill[fill=black!10] (0.5,0)coordinate (Q)  circle (1.1cm);
    \coordinate (a) at (0,0);
    \coordinate (b) at (1,0);

     \fill (a)  circle (2pt);
\fill[] (b) circle (2pt);
    \fill[white] (a) -- (b) -- ++(0,1.1) --++(-1,0) -- cycle;
 \draw  (a) -- (b);
 \draw (a) -- ++(0,1);
 \draw (b) -- ++(0,1);

\node[below] at (Q) {$2$};
    \end{scope}
\end{tikzpicture}
\caption{The differential of $\omoduli[1](4;-2,-2,-2)$ with residues equal to $(0,1,-1)$.} \label{fig:diffgcerocasob}
\end{figure}
 Since there is a unique class of prong-matching, there is a unique multi-scale differential $\omega$ with this underlying pointed curve. By Lemma~\ref{lem:gradosemiastable} in the case $a_{1}=a_{2}=1$ there exists a neighborhood of the boundary of  $\barmoduli[1,1]$ on which there are two differentials near~$\omega$. 
 Moreover, as recalled in Section~\ref{sec:msd}, in a neighborhood of this multi-scale differential the restriction of the differential on the top component looks like $\omega_{0}$ and on the bottom component like $t\omega_{-1}$. The curves are isomorphic for the parameters $t$ and $-t$ and hence  the differentials on these curves are distinct.
\par
We now consider the case pictured in the third column. Since the residues of the poles of order $-2$ vanish, Theorem~1.5 of \cite{geta} implies that $\kappa_{1}=\kappa_{2}=2$. By proposition~2.3 of \cite{chenchen2} there is a unique element in $\omoduli[0](1,1;-2,-2)$ such that the residues at the poles vanish. Moreover there are $2$ classes of prong-matching which give two disjoint families of multi-scale differentials. 
Locally at the nodes $n_{i}$ of the pointed curve, the first family $\calX_{1}$ is given by $x_{i}y_{i}=t$  and the second one $\calX_{2}$ is given by $x_{i}y_{i} = (-1)^{i}t$ for $i=1,2$.  
Moreover we know that in a neighborhood of the multi-scale differential the form on the top component looks like $\omega_{0}$ and on the bottom component like $t^{2}\omega_{-1}$. Hence the differentials for $t$ and for $-t$ are equals to each over. Hence on the isomorphic curves above $t$ and $-t$ both differentials are isomorphic. Hence there is a unique differential degenerating to each multi-scale differential on this pointed curve.
\par
We conclude that the number of distinct differentials in $\ores[1](4;-2,-2)$ on a smooth curve of genus $1$ is $1+2+2=5$.
\bigskip
 \par
 We now deal with the case of the locus $\ores[1](5;-3,-2)$. There are four different types of pointed curves on which there can exist a multi-scale differential in the limit of the locus $\ores[1](5;-3,-2)$. These pointed curves are pictured in the first line of Table~\ref{tab:5-3-2}.
  \par
 \begin{table}[h]
\begin{tabular}{|c|c|c|c|c|}
  \hline
  Pointed curve  & 
  \begin{tikzpicture}
\draw [] (-.5, 0) coordinate (x1)
  .. controls ++(0:1) and ++(0: .3) .. ( 0, -1) coordinate [pos=.5] (w2)  coordinate  (z)
  .. controls ++(180:.3) and ++(180:1) .. ( .5, 0) coordinate [pos=.5] (w1);
  \fill (z) circle (1pt);\node[below] at (z) {$z$};
  \fill (w1) circle (1pt);\node[left] at (w1) {$w_{1}$};
\fill (w2) circle (1pt);\node[right] at (w2) {$w_{2}$};
%
\end{tikzpicture}
& 
  \begin{tikzpicture}
\draw (-.5,-.2)  coordinate (q1) --  (.5,-.2) coordinate (q2)coordinate [pos=.5] (w2);
\draw (-.5,0) .. controls (-.3,-1.3) and (.3,-1.3) .. (.5,0) coordinate [pos=.3] (w1) coordinate [pos=.7] (z);

\node [above left] at (q1) {$\kappa_{1}$};\node [above right] at (q2) {$\kappa_{2}$};

\fill (z) circle (1pt); \node [below] at (z) {$z$};
\fill (w1) circle (1pt);\node [left] at (w1) {$w_{1}$};
\fill (w2) circle (1pt);\node [above] at (w2) {$w_{2}$};
\end{tikzpicture}
&
\begin{tikzpicture}
\draw (-.5,-.2)  coordinate (q1) --  (.5,-.2) coordinate (q2)coordinate [pos=.5] (w2);
\draw (-.5,0) .. controls (-.3,-1.3) and (.3,-1.3) .. (.5,0) coordinate [pos=.3] (w1) coordinate [pos=.7] (z);

\node [above left] at (q1) {$\kappa_{1}$};\node [above right] at (q2) {$\kappa_{2}$};

\fill (z) circle (1pt); \node [below] at (z) {$z$};
\fill (w1) circle (1pt);\node [left] at (w1) {$w_{2}$};
\fill (w2) circle (1pt);\node [above] at (w2) {$w_{1}$};
\end{tikzpicture}
&
  \begin{tikzpicture}
 \draw (-.8,-.2)  coordinate (q1) --  (.8,-.2) coordinate (q2)coordinate [pos=.35] (w1)coordinate [pos=.65] (w2);
\draw (-.8,0) .. controls (-.3,-1.3) and (.3,-1.3) .. (.8,0)  coordinate [pos=.5] (z);

\node [above left] at (q1) {$\kappa_{1}$};\node [above right] at (q2) {$\kappa_{2}$};

\fill (z) circle (1pt); \node [below] at (z) {$z$};
\fill (w1) circle (1pt);\node [above] at (w1) {$w_{1}$};
\fill (w2) circle (1pt);\node [above] at (w2) {$w_{2}$};
\end{tikzpicture}
\\
  \hline
  \#  twisted differentials & 4 & 2 & 1 & 2 \\
  \hline
  \#   prong-matching & 1 & 1 & 1 & 1 \\
  \hline
 Local degree & 1 & 2 & 3 & 5 \\
  \hline
   Symmetries & 1 & 1 & 1 & 1 \\
  \hline
  Total count & 4 & 4 & 6 & 10 \\
  \hline
\end{tabular}
\caption{Summary of the case $\ores[1](5;-3,-2)$.}
\label{tab:5-3-2}
\end{table}
 \par
 Consider the irreducible curve of the first column. In this case the differential has simple poles at the nodal points. So its pull-back lies in the stratum $\omoduli[1](5;-3,-2,-1,-1)$ and its residues at the non simple poles vanish. 
 By Proposition~3.8 of \cite{chenchen2} there are precisely $2$ non isomorphic such differentials represented in Figure~\ref{fig:5-3-2-1-1}. 
 \begin{figure}[htb]
 \center
\begin{tikzpicture}[scale=.73]
\begin{scope}[xshift=-.4cm,yshift=-.8cm]
\coordinate (a) at (-1,1);
\coordinate (b) at (0,1);

    \fill[fill=black!10] (a)  -- (b)coordinate[pos=.5](f) -- ++(0,1) --++(-1,0) -- cycle;
    \fill (a)  circle (2pt);
\fill[] (b) circle (2pt);
 \draw  (a) -- (b);
 \draw (a) -- ++(0,.9) coordinate (d)coordinate[pos=.5](h);
 \draw (b) -- ++(0,.9) coordinate (e)coordinate[pos=.5](i);
 \draw[dotted] (d) -- ++(0,.2);
 \draw[dotted] (e) -- ++(0,.2);
\node[above] at (f) {$1$};
\end{scope}

\begin{scope}[xshift=-.4cm,yshift=-1.2cm]
\coordinate (a) at (-1,1);
\coordinate (b) at (0,1);

\fill[fill=black!10] (a)  -- (b)coordinate[pos=.5](f) -- ++(0,-1) --++(-1,0) -- cycle;
\fill (a)  circle (2pt);
\fill[] (b) circle (2pt);
\draw  (a) -- (b);
\draw (a) -- ++(0,-.9) coordinate (d)coordinate[pos=.5](h);
\draw (b) -- ++(0,-.9) coordinate (e)coordinate[pos=.5](i);
\draw[dotted] (d) -- ++(0,-.2);
\draw[dotted] (e) -- ++(0,-.2);
\node[below] at (f) {$3$};
\end{scope}

\begin{scope}[xshift=1cm,yshift=0cm]
\fill[fill=black!10] (0,0) coordinate (Q) circle (1.1cm);
\coordinate (a) at (-.5,0);
\coordinate (b) at (.5,0);
\fill (a)  circle (2pt);
\fill (b) circle (2pt);
\draw (a) -- (b)coordinate[pos=.5](d);

\node[below] at (d) {$2$};
\node[above] at (d) {$3$};
\end{scope}

\begin{scope}[xshift=3.4cm,yshift=0cm]
\fill[fill=black!10] (0,0) coordinate (Q) circle (1.1cm);
\coordinate (a) at (-.5,0);
\coordinate (b) at (.5,0);
\fill (a)  circle (2pt);
\fill (b) circle (2pt);
\draw (a) -- (b)coordinate[pos=.5](d);
\draw (b) -- ++ (.6,0)coordinate[pos=.5](e);
\node[below] at (d) {$1$};
\node[above] at (d) {$2$};
\node[below] at (e) {$5$};
\node[above] at (e) {$4$};
\end{scope}

\begin{scope}[xshift=6cm,yshift=0cm]
\fill[fill=black!10] (0,0) coordinate (Q) circle (1cm);
\coordinate (a) at (0,0);
\fill (a)  circle (2pt);
\draw (a) -- ++ (1,0)coordinate[pos=.5](d);

\node[below] at (d) {$4$};
\node[above] at (d) {$5$};
\end{scope}

\begin{scope}[xshift=11cm]
 \begin{scope}[xshift=-.4cm,yshift=-.8cm]
\coordinate (a) at (-1,1);
\coordinate (b) at (0,1);

    \fill[fill=black!10] (a)  -- (b)coordinate[pos=.5](f) -- ++(0,1) --++(-1,0) -- cycle;
    \fill (a)  circle (2pt);
\fill[] (b) circle (2pt);
 \draw  (a) -- (b);
 \draw (a) -- ++(0,.9) coordinate (d)coordinate[pos=.5](h);
 \draw (b) -- ++(0,.9) coordinate (e)coordinate[pos=.5](i);
 \draw[dotted] (d) -- ++(0,.2);
 \draw[dotted] (e) -- ++(0,.2);
\node[above] at (f) {$1$};
\end{scope}

\begin{scope}[xshift=-.4cm,yshift=-1.2cm]
\coordinate (a) at (-1,1);
\coordinate (b) at (0,1);

\fill[fill=black!10] (a)  -- (b)coordinate[pos=.5](f) -- ++(0,-1) --++(-1,0) -- cycle;
\fill (a)  circle (2pt);
\fill[] (b) circle (2pt);
\draw  (a) -- (b);
\draw (a) -- ++(0,-.9) coordinate (d)coordinate[pos=.5](h);
\draw (b) -- ++(0,-.9) coordinate (e)coordinate[pos=.5](i);
\draw[dotted] (d) -- ++(0,-.2);
\draw[dotted] (e) -- ++(0,-.2);
\node[below] at (f) {$3$};
\end{scope}

\begin{scope}[xshift=1cm,yshift=0cm]
\fill[fill=black!10] (0,0) coordinate (Q) circle (1.1cm);
\coordinate (a) at (-.5,0);
\coordinate (b) at (.5,0);
\fill (a)  circle (2pt);
\fill (b) circle (2pt);
\draw (a) -- (b)coordinate[pos=.5](d);

\node[below] at (d) {$2$};
\node[above] at (d) {$3$};

\end{scope}

\begin{scope}[xshift=3.4cm,yshift=0cm]
\fill[fill=black!10] (0,0) coordinate (Q) circle (1.1cm);
\coordinate (a) at (-.5,0);
\coordinate (b) at (.5,0);
\fill (a)  circle (2pt);
\fill (b) circle (2pt);
\draw (a) -- (b)coordinate[pos=.5](d);

\node[below] at (d) {$1$};
\node[above] at (d) {$2$};

\draw (a) -- ++ (-.6,0)coordinate[pos=.5](d);

\node[below] at (d) {$5$};
\node[above] at (d) {$4$};
\end{scope}

\begin{scope}[xshift=6cm,yshift=0cm]
\fill[fill=black!10] (0,0) coordinate (Q) circle (1cm);
\coordinate (a) at (0,0);
\fill (a)  circle (2pt);
\draw (a) -- ++ (-1,0)coordinate[pos=.5](d);

\node[below] at (d) {$4$};
\node[above] at (d) {$5$};
\end{scope}
\end{scope}

\end{tikzpicture}
\caption{The two differentials of $\omoduli[1](5;-3,-2,-1,-1)$ with residues equal to $(0,0,1,-1)$.} \label{fig:5-3-2-1-1}
\end{figure}
 Moreover the simple poles can be distinguished by the fact that one is next to the pole of order $-2$ and the other next to the pole of order~$-3$. Hence we obtain $4$ such multi-scale differentials. It is easy to see that each of these differentials can be smoothed in a unique way.
 \par
 Now consider the curve in the second column of Table~\ref{tab:5-3-2}, when the pole of order~$-2$ is on the top component. The differential on the top component is the unique differential in $\omoduli[0](0,0;-2)$. The  differential on the bottom component lies in the meromorphic stratum $\omoduli[0](5;-3,-2,-2)$ with a residue equal to $0$ at the pole of order~$-3$. Again \cite[Proposition 3.8]{chenchen2} implies that there are $2$ non isomorphic such differentials. As before the local degree of each family  is $2$ and there are two distinct differential on a smooth curve which degenerate to the considered pointed curve.
 \par
 Now consider the curve in the third column of Table~\ref{tab:5-3-2}, when the pole of order~$-3$ is on the top component. The differential on the top component in the stratum $\omoduli[0](1,0;-3)$. The differential on the bottom component is in the stratum $\omoduli[0](5;-3,-2,-2)$ with a residue equal to $0$ at the marked pole of order~$-2$. There is a unique such differential. Note that the prong numbers are equal to $1$ and $2$. Hence there are two such \twd depending on which node we put the distinct prong numbers. There is only one class of prong-matching. Since locally at the nodes  the families are given by the equations $x_{1}y_{1} = t^{2}$ and $x_{1}y_{1} = t$, we deduce from Lemma~\ref{lem:gradosemiastable} that each family is a cover of order $3$ to $\moduli[1,1]$. Hence there are $6$ differentials on a smooth curve which arise from this case.
 \par
 Finally, we consider the fouth pointed curve where both poles are on the top component. The differential on this component is the locus $\ores[0](2,1;-2,-3)$. Again \cite{chenchen2} gives that there is a unique such differential. The differential  on the bottom component is the unique differential in $\omoduli[0](5;-2,-3)$. Since there are two ways to assign the prong numbers at the nodes,  there are two \twds on this pointed curve. Note that there is a unique class of prong-matching, hence there are two distinct multi-scale differentials on this pointed curve. Again each family is locally a cover of order $5$ to~$\moduli[1,1]$. Hence there are $10$ differentials which come from this case.
\par
Summing up all the contributions we conclude that there are~$24$ non isomorphic differentials in $\ores[1](5;-2,-3)$ on a smooth curve of genus~$1$. 
\end{proof}
\par

\smallskip
\par
We now prove Theorem~\ref{thm:degreeprojs} by  arguments similar to the ones used to prove Lemma~\ref{lem:gradoexs}. Recall that this theorem says that 
 the degree of $\pi\colon \PP\ores[1](6;-2,-2,-2) \to \moduli[1,1]$ is~$7$. Note that by Theorem~\ref{thm:dimprojintro}, the dimension of $\PP\ores[1](6;-2,-2,-2)$ is equal to the dimension of~$\moduli[1,1]$, hence it make sense to compute the degree of the forgetful map. 

\begin{proof}[Proof of Theorem \ref{thm:degreeprojs}]
 According to Lemma~\ref{lem:degenposgun} there are $4$ types of pointed curves on which there can exists a multi-scale differential in the closure of $\ores[1](6;-2,-2,-2)$. These curves are shown in the first row of Table~\ref{tab:6-2-2-2} (the entries of this table are explained after Table~\ref{tab:4-2-2}).
 \par
 \begin{table}[h]
\begin{tabular}{|c|c|c|c|c|}
  \hline
  Pointed curve  & 
  \begin{tikzpicture}
\draw [] (-.5, 0) coordinate (x1)
  .. controls ++(0:1) and ++(0: .3) .. ( 0, -1) coordinate [pos=.5] (w2) coordinate [pos=.7] (w3) coordinate  (z)
  .. controls ++(180:.3) and ++(180:1) .. ( .5, 0) coordinate [pos=.5] (w1);
  \fill (z) circle (1pt);\node[below] at (z) {$z$};
  \fill (w1) circle (1pt);\node[left] at (w1) {$w_{1}$};
\fill (w2) circle (1pt);\node[right] at (w2) {$w_{2}$};
\fill (w3) circle (1pt);\node[right] at (w3) {$w_{3}$};
%
\end{tikzpicture}
& 
  \begin{tikzpicture}
\draw (-.45,-.2)  coordinate (q1) --  (.45,-.2) coordinate (q2)coordinate [pos=.5] (w2);
\draw (-.45,0) .. controls (-.3,-1.3) and (.3,-1.3) .. (.45,0) coordinate [pos=.25] (w1) coordinate [pos=.5] (z)coordinate [pos=.75] (w3);

\node [above left] at (q1) {$\kappa_{1}$};\node [above right] at (q2) {$\kappa_{2}$};

\fill (z) circle (1pt); \node [below] at (z) {$z$};
\fill (w1) circle (1pt);\node [left] at (w1) {$w_{1}$};
\fill (w2) circle (1pt);\node [above] at (w2) {$w_{2}$};
\fill (w3) circle (1pt);\node [right] at (w3) {$w_{3}$};
\end{tikzpicture}
&
  \begin{tikzpicture}
 \draw (-.7,-.2)  coordinate (q1) --  (.7,-.2) coordinate (q2)coordinate [pos=.35] (w1)coordinate [pos=.65] (w2);
\draw (-.7,0) .. controls (-.3,-1.3) and (.3,-1.3) .. (.7,0)  coordinate [pos=.5] (z)coordinate [pos=.75] (w3);

\node [above left] at (q1) {$\kappa_{1}$};\node [above right] at (q2) {$\kappa_{2}$};

\fill (z) circle (1pt); \node [below] at (z) {$z$};
\fill (w1) circle (1pt);\node [above] at (w1) {$w_{1}$};
\fill (w2) circle (1pt);\node [above] at (w2) {$w_{2}$};
\fill (w3) circle (1pt);\node [right] at (w3) {$w_{3}$};
\end{tikzpicture}
&
  \begin{tikzpicture}
 \draw (-.9,-.2)  coordinate (q1) --  (.9,-.2) coordinate (q2)coordinate [pos=.25] (w1) coordinate [pos=.5] (w2) coordinate [pos=.75] (w3);
\draw (-.9,0) .. controls (-.3,-1.3) and (.3,-1.3) .. (.9,0)  coordinate [pos=.5] (z);

\node [above left] at (q1) {$\kappa_{1}$};\node [above right] at (q2) {$\kappa_{2}$};

\fill (z) circle (1pt); \node [below] at (z) {$z$};
\fill (w1) circle (1pt);\node [above] at (w1) {$w_{1}$};
\fill (w2) circle (1pt);\node [above] at (w2) {$w_{2}$};
\fill (w3) circle (1pt);\node [above] at (w3) {$w_{3}$};
\end{tikzpicture}
\\
  \hline
  \#  twisted differentials & 1 & 1 & 1 & 1 \\
  \hline
  \#   prong-matching & 1 & 1 & 2 & 3 \\
  \hline
 Local degree & 1 & 2 &  2 & 2 \\
  \hline
   Symmetries & 1 & 1 & 2 & 3 \\
  \hline
  Total count & 1 & 2 & 2 & 2 \\
  \hline
\end{tabular}
\caption{Summary of the case $\ores[1](6;-2,-2,-2)$.}
\label{tab:6-2-2-2}
\end{table}
 \par
 Let us consider the pointed curves of the first column. The pull-back of the differential on the normalisation lies  in  the stratum $\omoduli[0](6;-2,-2,-2,-1,-1)$ with all the residues of the poles of order $-2$ equal to zero. There is a unique such differential, and it is not difficult to show that it leads to  a unique multi-scale differential on such pointed curve. Hence there is a unique differential on a smooth curve degenerating to this multi-scale differential.
 \par
 Consider the pointed curve in the second column. The differential on the top component is the unique one in the stratum $\omoduli[0](0,0;-2)$. The differential on the bottom component is in $\omoduli[0](6;-2,-2,-2,-2)$ with two poles having zero residue. There is a unique such \twd by \cite{chenchen2}. Moreover, since there is a unique class of prong-matching, there is a unique multi-scale differential on this pointed curve. The local equation of the family at the nodes is $x_{i}y_{i}=t$, so by Lemma~\ref{lem:gradosemiastable} the local degree of the family of pointed curves is~$2$. Since the family of differentials near the top component is given by $\omega_{0}$ and the bottom by $t \omega_{-1}$, the differentials on the isomorphic curves for parameters $t$ and $-t$ are distinct. Hence there are two differentials on a smooth genus $1$ curves converging to this multi-scale differential.
 \par
 Consider the case of pointed curves given in the third column of Table~\ref{tab:6-2-2-2}. The differential on the top component is the unique differential in $\ores[0](1,1;-2,-2)$. The differential on the bottom component is in the stratum $\omoduli[0](6;-2,-3,-3)$ with two poles having zero residue. Proposition~2.3 of \cite{chenchen2} implies that there is a unique such twisted differential. Since the prong-numbers are both equal to $2$, there are $2$ classes of prong-matchings. For each class the degree of the restriction of the stabilisation map to this family is~$2$. But since near the bottom differential the family looks like $t^{2}\omega_{-1}$, the $2$ differentials coincide (the marked poles are permuted). Hence there are $2$ differentials on a smooth curve which degenerate to the multi-scale differentials in this case.
 \par
 In the last pointed curve of Table~\ref{tab:6-2-2-2}, the restriction $\omega_{0}$ of the differential to the top component is the unique element in the locus $\ores[0](2,2;-2,-2,-2)$. Note that the top differential in $\ores[0](2,2;-2,-2,-2)$ has an automorphism group equal to $\ZZ/3\ZZ$. This can be done by checking directly that $\omega_{0} = \tfrac{z^{2}}{(z^{3}-1)^{2}} dz$ or showing at the flat picture of it given in Figure~\ref{fig:classpm}. The differential on the bottom component is the unique element in $\omoduli[0](6;-4,-4)$. There is a unique such twisted differential but there are $3$ classes of prong-matchings. Here we have a new phenomenon occurring. Indeed the automorphism group of $\omega_{0}$ acts on the classes of prongs and hence on the multi-scale differentials. One class of prongs is represented in Figure~\ref{fig:classpm} and the acts of  the automorphism group consist of permuting cyclically the polar domains. Hence there is only one class prong-matching modulo isomorphism.  For the smoothing of this multi-scale differential the degree of the stabilisation map is~$2$. It is easy to see that the $2$ differentials on isomorphic curves are indeed distinct. Hence there are $2$  differentials on a smooth curve degenerating to this multi-scale differential.
 \begin{figure}[htb]
 \center
\begin{tikzpicture}
\begin{scope}[xshift=0cm]
\fill[fill=black!10] (0,0) coordinate (Q) circle (1.1cm);
\coordinate (a) at (0,-.5);
\coordinate (b) at (0,.5);
\draw[->] (a) -- ++(.3,0);
\draw[->] (b) -- ++(.3,0);
\draw (a) -- (b)coordinate[pos=.5](d);
\filldraw[fill=white] (a)  circle (2pt);
\fill (b) circle (2pt);

\node[left] at (d) {$1$};
\node[right] at (d) {$2$};
    \end{scope}

\begin{scope}[yshift=-2.5cm]
\fill[fill=black!10] (0,0) coordinate (Q) circle (1.1cm);
\coordinate (a) at (0,-.5);
\coordinate (b) at (0,.5);
\draw (a) -- (b)coordinate[pos=.5](d);
\filldraw[fill=white] (a)  circle (2pt);
\fill (b) circle (2pt);

\node[left] at (d) {$3$};
\node[right] at (d) {$1$};
\end{scope}

\begin{scope}[yshift=-5cm]
\fill[fill=black!10] (0,0) coordinate (Q) circle (1.1cm);
\coordinate (a) at (0,-.5);
\coordinate (b) at (0,.5);
\draw (a) -- (b)coordinate[pos=.5](d);
\filldraw[fill=white] (a)  circle (2pt);
\fill (b) circle (2pt);

\node[left] at (d) {$2$};
\node[right] at (d) {$3$};
    \end{scope}

    \begin{scope}[xshift=5cm]
\begin{scope}[xshift=0cm]
\fill[fill=black!10] (0,0) coordinate (Q) circle (1.1cm);
\coordinate (a) at (0,-.5);
\coordinate (b) at (0,.5);
\draw (a) -- (b)coordinate[pos=.5](d);
\filldraw[fill=white] (a)  circle (2pt);
\fill (b) circle (2pt);

\node[left] at (d) {$1$};
\node[right] at (d) {$2$};
    \end{scope}

\begin{scope}[yshift=-2.5cm]
\fill[fill=black!10] (0,0) coordinate (Q) circle (1.1cm);
\coordinate (a) at (0,-.5);
\coordinate (b) at (0,.5);
\draw[->] (b) -- ++(.3,0);
\draw (a) -- (b)coordinate[pos=.5](d);
\filldraw[fill=white] (a)  circle (2pt);
\fill (b) circle (2pt);

\node[left] at (d) {$3$};
\node[right] at (d) {$1$};
\end{scope}

\begin{scope}[yshift=-5cm]
\fill[fill=black!10] (0,0) coordinate (Q) circle (1.1cm);
\coordinate (a) at (0,-.5);
\coordinate (b) at (0,.5);
\draw[->] (a) -- ++(.3,0);
\draw (a) -- (b)coordinate[pos=.5](d);
\filldraw[fill=white] (a)  circle (2pt);
\fill (b) circle (2pt);

\node[left] at (d) {$2$};
\node[right] at (d) {$3$};
    \end{scope}
    \end{scope}

        \begin{scope}[xshift=10cm]
\begin{scope}[xshift=0cm]
\fill[fill=black!10] (0,0) coordinate (Q) circle (1.1cm);
\coordinate (a) at (0,-.5);
\coordinate (b) at (0,.5);
\draw (a) -- (b)coordinate[pos=.5](d);
\filldraw[fill=white] (a)  circle (2pt);
\fill (b) circle (2pt);

\node[left] at (d) {$1$};
\node[right] at (d) {$2$};
    \end{scope}

\begin{scope}[yshift=-2.5cm]
\fill[fill=black!10] (0,0) coordinate (Q) circle (1.1cm);
\coordinate (a) at (0,-.5);
\coordinate (b) at (0,.5);
\draw[->] (a) -- ++(.3,0);
\draw (a) -- (b)coordinate[pos=.5](d);
\filldraw[fill=white] (a)  circle (2pt);
\fill (b) circle (2pt);

\node[left] at (d) {$3$};
\node[right] at (d) {$1$};
\end{scope}

\begin{scope}[yshift=-5cm]
\fill[fill=black!10] (0,0) coordinate (Q) circle (1.1cm);
\coordinate (a) at (0,-.5);
\coordinate (b) at (0,.5);
\draw[->] (b) -- ++(.3,0);
\draw (a) -- (b)coordinate[pos=.5](d);
\filldraw[fill=white] (a)  circle (2pt);
\fill (b) circle (2pt);

\node[left] at (d) {$2$};
\node[right] at (d) {$3$};
    \end{scope}
    \end{scope}
\end{tikzpicture}
\caption{The three element in the same class of prong-matchings in the case $\ores[1](6;-2,-2,-2)$.} \label{fig:classpm}
\end{figure}
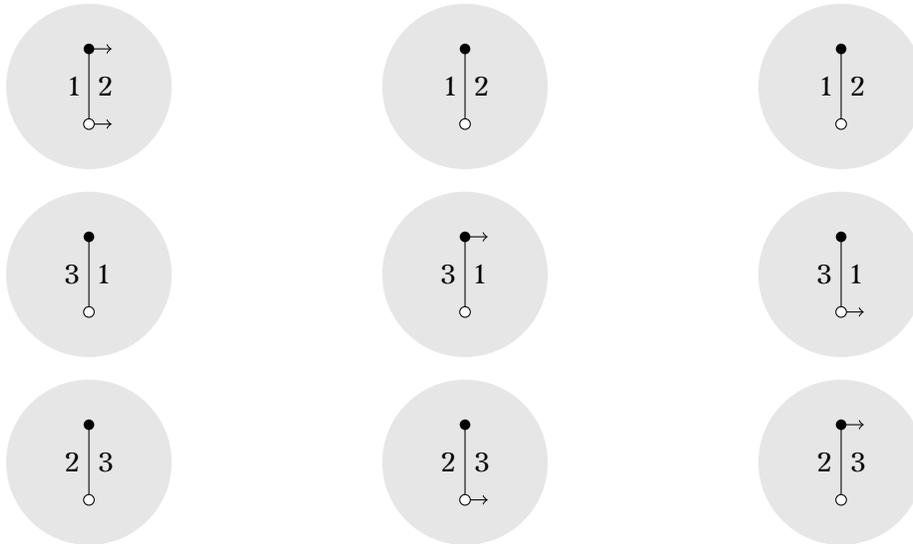
 
 \par
 Summing up all the cases we obtain that there are $7$ non isomorphic differentials in the locus $\PP\ores[1](6;-2,-2,-2)$ on a general curve of genus~$1$.
 \smallskip
 \par
 To conclude the proof, it suffices to show that the map $\pi\colon \PP\ores[1](6;-2,-2,-2) \to \moduli[1,1]$ is unramified. To prove this, it suffices to show that there exit no twisted differentials of type $(6;-2,-2,-2)$ of the second kind on a curve semi-stably equivalent to a smooth curve of genus~$1$. This follows from \cite[Theorem 1.1]{muHur} by the fact that every twisted differential on such curve has a non-zero residue either at a node or at a smooth pole.
\end{proof}

To conclude this section on enumerative problems related to differentials, we compute the degree of the map $\pi\colon\PP\ores[2](6;-2,-2)\to\moduli[2]$, which is finite by Theorem~\ref{thm:dimprojintro}.
\begin{prop}\label{prop:gradg2res0}
The degree of the map $\pi\colon\PP\ores[2](6;-2,-2)\to\moduli[2]$ is~$644$.
\end{prop}
The proof of this proposition is by degeneration in the moduli space of multi-scale differentials. 

\begin{proof}
 Let us consider the stable curve~$X$ which is union of two general elliptic curves~$X_{1}$ and $X_{2}$ attached to a point~$q$. We denote by $z$ the zero of order $a$, by $w_{1}$ and by $w_{2}$ the poles of order~$2$. We compute  the number of multi-scale differentials on $X$.
 \par
There are six types of pointed curves $(X;z,w_{1},w_{2})$ which can have a  multi-scale differential of type $(6;-2,-2)$ of the second kind. The ones with $z\in X_{1}$  are shown in Figure~\ref{fig:rescero} and the other stable curves are symmetric to these ones  with $z\in X_{2}$.

 \begin{figure}[ht]
 \centering
\begin{tikzpicture}[scale=1.2]
\begin{scope}[xshift=-3cm]
\draw (-4,0.1) coordinate (x1) .. controls (-3.4,-.3) and (-3.6,-.6) .. (-3,-1.1) coordinate (q1) coordinate [pos=.15] (z)coordinate [pos=.45] (p1)coordinate [pos=.7] (p2);
\fill (z) circle (1pt); \node[right] at (z) {$z$};
\fill (p1) circle (1pt); \node[right] at (p1) {$w_{1}$};
\fill (p2) circle (1pt); \node[right] at (p2) {$w_{2}$};
\node[above] at (x1) {$X_{1}$};
\draw (-3.1,-1.1) coordinate (x1) .. controls (-2.5,-.6) and (-2.7,-.3) .. (-2.1,.1) coordinate (q2);
\node[above] at (q2) {$X_{2}$};

\node at (-3.1,-1.4) {(a)};
\end{scope}

\begin{scope}[xshift=0cm]
\draw (-4,0.1) coordinate (x1) .. controls (-3.4,-.3) and (-3.6,-.6) .. (-3,-1.1) coordinate (q1) coordinate [pos=.25] (z)coordinate [pos=.65] (p1);
\fill (z) circle (1pt); \node[right] at (z) {$z$};
\fill (p1) circle (1pt); \node[right] at (p1) {$w_{1}$};
\node[above] at (x1) {$X_{1}$};
\draw (-3.1,-1.1) coordinate (x1) .. controls (-2.5,-.6) and (-2.7,-.3) .. (-2.1,.1) coordinate (q2) coordinate [pos=.5] (p2);
\node[above] at (q2) {$X_{2}$};
\fill (p2) circle (1pt); \node[right] at (p2) {$w_{2}$};

\node at (-3.1,-1.4) {(b)};
\end{scope}

\begin{scope}[xshift=3cm]
\draw (-4,0.1) coordinate (x1) .. controls (-3.4,-.3) and (-3.6,-.6) .. (-3,-1.1) coordinate (q1) coordinate [pos=.5] (z);
\fill (z) circle (1pt); \node[right] at (z) {$z$};
\node[above] at (x1) {$X_{1}$};
\draw (-3.1,-1.1) coordinate (x1) .. controls (-2.5,-.6) and (-2.7,-.3) .. (-2.1,.1) coordinate (q2) coordinate [pos=.3] (p1) coordinate [pos=.7] (p2);
\node[above] at (q2) {$X_{2}$};
\fill (p1) circle (1pt); \node[right] at (p1) {$w_{1}$};
\fill (p2) circle (1pt); \node[right] at (p2) {$w_{2}$};
\node at (-3.1,-1.4) {(c)};
\end{scope}

\end{tikzpicture}
\caption{The pointed curves $(X,z,w_{1},w_{2})$ underlying multi-scale differentials in the closure of  $\ores[2](6;-2,-2)$ (up to permutation of the points).}
\label{fig:rescero}
\end{figure}
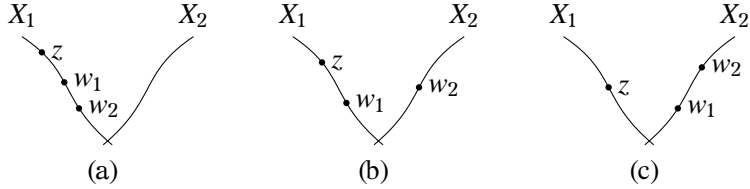
\par
In case (a), the differential $\omega_{1}$ on $X_{1}$ is in the locus  $\ores[1](6;-2,-2,-2)$. Theorem~\ref{thm:degreeprojs} gives that there are $7$ non isomorphic such differentials. Moreover, the differential $\omega_{2}$ on~$X_{2}$ is the unique holomorphic differential on this elliptic curve. Since there are three ways of pasting these differentials together (one for each pole  of the  differentials~$\omega_{1}$), then there are $d_{a}=21$  such \twds. Moreover since there is a unique class of prong-matchings there are $21$ multi-scale differentials of this type.
\par
In  case (b), the differential over $X_{1}$ is in the locus $\ores[1](6;-2,-4)$. Corollary~\ref{cor:res0g1prim} gives that there are 
 $6^{2}+4^{2}-10=42$ such differentials.
  Moreover the differential on~$X_{2}$ is one of the  $3$ differentials of $\omoduli[1](2;-2)$. Since there is a unique class of prong-matching, there exist $d_{b}=42\cdot  3=126$ such multi-scale differentials.
\par
In  case (c), the differential  on $X_{1}$ in  one of the $35$ differentials of $\omoduli[1](6;-6)$.  The differential on $X_{2}$ is in $\ores[1](4;-2,-2)$. Hence Lemma~\ref{lem:gradoexs} implies that there are $5$ such differentials. Finally there exist $d_{c}=35 \cdot 5=175$ such multi-scale differentials.
\par
Summing up all the contributions, there exist
\[d = 2\cdot (d_{a} + d_{b} +d_{c})= 2\cdot( 21 + 126 + 175) =644 \]
multi-scale differentials in the closure of $\ores[2](6;-2,-2)$ on the curve~$X$.
 \par
 To conclude that  the degree of $\pi$ on $X$ is equal to~$644$, it remains to show that there is a unique differential which converges to each multi-scale differential. Since the family of curve is given by $xy = t$ at the node, there is a unique differential on nearby smooth curves which converges to each multi-scale differentials. 
\end{proof}

\subsection{The class of $\barmoduli[3](6;-2)$ in  the Picard group of $\barmoduli[3]$.}
\label{sec:compu}

To compute the boundary classes of the divisor $\barmoduli[3](6,-2)$ of the projection of $\omoduli[3](6,-2)$ in $\moduli[3]$, we use the technique of test curves.  The test curves that we use are well-known and we follow the notation of \cite{mudiv} to denote them.

\smallskip
\par
\paragraph{\bf First curve.}

The first curve $A$ is obtained in the following way. Let $X_{2}$ be a curve of genus $2$ and $p$ a generic point of $X_{2}$. We glue at $p$ a base point $q$ of a pencil of plane cubic. 
\begin{prop}\label{prop:interA}
 The intercession number of $\barmoduli[3](6;-2)$ with $ A$ is equal to $0$.
\end{prop}

\begin{proof}
 It is sufficient to show that there is no twisted differential of type $(6;-2)$ on a curve semi-stably equivalent to any curve of~$A$.  Lemma~\ref{lem:passurnoeud} tells us that it is enough to consider cases where the points $z$ and $w$ are on the smooth part of $X$. So we have four cases to consider, which we represent in Figure~\ref{fig:tipoAC}.
   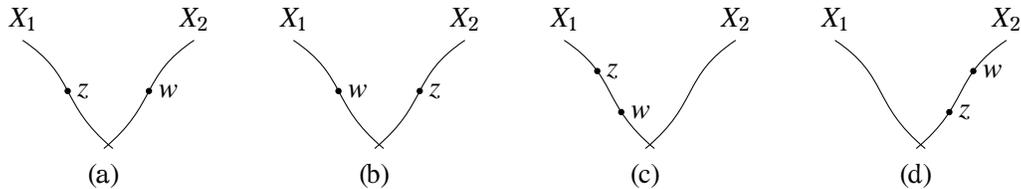
\begin{figure}[ht]
 \centering
\begin{tikzpicture}[scale=1.2]
\begin{scope}[xshift=-3cm]
\draw (-4,0.1) coordinate (x1) .. controls (-3.4,-.3) and (-3.6,-.6) .. (-3,-1.1) coordinate (q1) coordinate [pos=.5] (z1);
\fill (z1) circle (1pt); \node[right] at (z1) {$z$};
\node[above] at (x1) {$X_{1}$};
\draw (-3.1,-1.1) coordinate (x1) .. controls (-2.5,-.6) and (-2.7,-.3) .. (-2.1,.1) coordinate (q2) coordinate [pos=.5] (z1);
\node[above] at (q2) {$X_{2}$};
\fill (z1) circle (1pt);\node[right] at (z1) {$w$};

\node at (-3.1,-1.4) {(a)};
\end{scope}

\draw (-4,0.1) coordinate (x1) .. controls (-3.4,-.3) and (-3.6,-.6) .. (-3,-1.1) coordinate (q1) coordinate [pos=.5] (z1);
\fill (z1) circle (1pt); \node[right] at (z1) {$w$};
\node[above] at (x1) {$X_{1}$};
\draw (-3.1,-1.1) coordinate (x1) .. controls (-2.5,-.6) and (-2.7,-.3) .. (-2.1,.1) coordinate (q2) coordinate [pos=.5] (z1);
\node[above] at (q2) {$X_{2}$};
\fill (z1) circle (1pt);\node[right] at (z1) {$z$};

\node at (-3.1,-1.4) {(b)};

\begin{scope}[xshift=3cm]
\draw (-4,0.1) coordinate (x1) .. controls (-3.4,-.3) and (-3.6,-.6) .. (-3,-1.1) coordinate (q1) coordinate [pos=.3] (z1)coordinate [pos=.7] (z2);
\fill (z1) circle (1pt); \node[right] at (z1) {$z$};
\node[above] at (x1) {$X_{1}$};
\draw (-3.1,-1.1) coordinate (x1) .. controls (-2.5,-.6) and (-2.7,-.3) .. (-2.1,.1) coordinate (q2) ;
\node[above] at (q2) {$X_{2}$};
\fill (z2) circle (1pt);\node[right] at (z2) {$w$};

\node at (-3.1,-1.4) {(c)};
\end{scope}

\begin{scope}[xshift=6cm]
\draw (-4,0.1) coordinate (x1) .. controls (-3.4,-.3) and (-3.6,-.6) .. (-3,-1.1) coordinate (q1) ;
\node[above] at (x1) {$X_{1}$};
\draw (-3.1,-1.1) coordinate (x1) .. controls (-2.5,-.6) and (-2.7,-.3) .. (-2.1,.1) coordinate (q2) coordinate [pos=.3] (z1)coordinate [pos=.7] (z2);
\fill (z1) circle (1pt); \node[right] at (z1) {$z$};
\node[above] at (q2) {$X_{2}$};
\fill (z2) circle (1pt);\node[right] at (z2) {$w$};

\node at (-3.1,-1.4) {(d)};
\end{scope}
\end{tikzpicture}
\caption{The pointed curves $(X;z,w)$ which can both underly a multi-scale differential of type  $(6;-2)$ and be an element of~$A$.}
\label{fig:tipoAC}
\end{figure}
\par
In case (a) the restriction $\omega_{2}$ of $\omega$ to $X_{2}$ is in the stratum $\omoduli[2](4;-2)$. Since this stratum has a projective dimension equal to the dimension of $\moduli[2]$, there is a finite number of points that can be the pole of~$\omega_{2}$. Hence no twisted differential of this form can exist, since the point $p$ is generic.
\par
In case (b), the same argument works since the differential  $\omega_{2}$ is in $\omoduli[2](6;-4)$ whose projective dimension is equal to the one of $\moduli[2]$.
\par
In the case (c), the differential $\omega_{2}$ has a zero of order $2$ at $p$. So this point has to be a \Weierstrass point. It is not possible because this point is generic on $X_{2}$.
\par
In the case (d), the differential $\omega_{2}$ is in the stratum $\omoduli[2](6;-2,-2)$. In addition, the global residue condition implies that all the residues are equal to~$0$. By Lemma~\ref{lm:dimres} the locus that parametrises these differentials is of dimension three. This implies that the point $p$ has to be special on $X_{2}$, which gives the last impossibility.
\end{proof}

\smallskip
\par
\paragraph{\bf Second curve.}

The second curve $C$ is obtained as follows. We choose a general curve~$X_{2}$ of genus $2$ and an elliptic curve $X_{1}$ which we paste together at points $q_{2}\in X_{2}$ and $q_{1}\in X_{1}$. The family $C$  is the family where $q_{2}$ varies over $X_{2}$.
\begin{prop}\label{prop:interC}
 The intersection number between the divisor $\barmoduli[3](6;-2)$ and the curve~$C$ is equal to~$8792$.
\end{prop}

\begin{proof}
 A curve $X$ in the family $C$ is in the locus $\barmoduli[3](6;-2)$ if and only if there is a multi-scale differential of type $(6;-2)$ on a pointed curve semi-stably equivalent to~$X$. Lemma~\ref{lem:passurnoeud} implies that no marked point can be on a rational bridge between the two components $X_{1}$ and~$X_{2}$. Moreover, the points cannot be together on a rational curve. Otherwise on this curve the differential has to be in the stratum $\omoduli[0](6;-2,-6)$. By \cite[Theorem 1.5]{geta} implies that the residues of the differential at the poles are not zero. Then the global residue condition cannot be fulfilled. Therefore, the possible curves belonging to $C$ and in the divisor $\barmoduli[3](6;-2)$ are of the types pictured in Figure~\ref{fig:tipoAC}.
 \smallskip
\par
In case (a), the differential over $X_{1}$ is in the stratum $\omoduli[1](6;-6)$ and the differential on $X_{2}$ is in $\omoduli[2](4;-2)$. By Proposition~\ref{prop:gradg2} there are $d_{2}=2\cdot 4^{2} \cdot (-2)^{2} -18 = 110$ possible position for the points $q_{2}$. So there are $110$ curves in the intersection of $C$ and $\barmoduli[3](6;-2)$ where the zero tends to $X_{1}$ and the pole to $X_{2}$. It remains to compute the multiplicity of these points. Note that on $X_{1}$, there are  $d_{1}=6^{2}-1 = 35$ possible non isomorphic differentials. Moreover the space of multi-scale differentials is smooth at these points. Hence this configuration contribute to $d_{a}=d_{1}\cdot d_{2} = 3850$ to the intersection number between $C$ and $\barmoduli[3](6;-2)$.
\par
In case (b), the restriction of the multi-scale differential to $X_{1}$ is in $\omoduli[1](2;-2)$ and the restriction to $X_{2}$ is in $\omoduli[2](6;-4)$. So by Proposition~\ref{prop:gradg2} there is $d_{2}=2\cdot 6^{2} \cdot 4^{2} -18 = 1134$ possible positions for the point~$q_{2}$. Hence there are $1134$ curves of $C$ having multi-scale differentials of this type. Moreover, since there are $d_{1}=3$ possible differentials on $X_{1}$ and the space is smooth at these points, each of this point contribute to $3$ to the intersection number. So this case contribute to $d_{b}=d_{1}\cdot d_{2} =3402 $ to the intersection number  between~$C$ and $\barmoduli[3](6;-2)$.
\par
In case (c), the restriction of the multi-scale differential to $X_{1}$ is in $\omoduli[1](6;-2,-4)$ and the restriction to $X_{2}$ is in $\omoduli[2](2)$. So there are $6$ pointed curves that have a twisted differential of this type (one for each \Weierstrass point). Moreover, Corollary~\ref{cor:res0g1prim} tells us that there are $6^{2} + 4^{2} -10=42$ differentials in $\ores[1](6;-2,-4)$ over $X_{1}$. So this case contribute to  $d_{b}=6*42=252$ to the intersection number.
\par
In case (d), the restriction of the multi-scale differential to $X_{1}$ is in $\omoduli[1]$ and the restriction to $X_{2}$ is in $\ores[2](6;-2,-2)$. So Proposition~\ref{prop:gradg2res0} implies that the number of curves in $C$  that lies in $\barmoduli[3](6;-2)$ with a multi-scale differential of type~(d) is $644$. Since each of the intersection is simple and there is two choices for the pole of order $2$, this case contribute to  $1288$ to the intersection number.
\smallskip
\par
To conclude the proof, it is sufficient to add up the contributions of each type. Hence we have
\begin{equation*}
  \barmoduli[3](6;-3) \cap \left[ C \right]  = 3850 + 3402 + 252 + 1288 = 8792 \,.
\end{equation*}
\end{proof}

\smallskip
\par
\paragraph{\bf Conclusion.}

Now it remains to collect the information of this section to give the class $\alpha\lambda + \beta\delta_{0} + \gamma\delta_{1}$ of $\barmoduli[3](6;-2)$ in $\Pic(\barmoduli[3])\otimes\bQ$.  The classical result of \cite[Table~3.141]{hamo} implies that:  
\begin{eqnarray*}
     \alpha +12 \beta - \gamma &=& 0\, ,\\
      -2\gamma &=& 8792 \, .
\end{eqnarray*}
Moreover we now by Equation~\eqref{eq:lambdaclasse} that $\alpha=17108$. 
Hence the class of $ \barmoduli[3](6;-2)$ in  
$\Pic(\barmoduli[3])\otimes\bQ$ is given by

\begin{equation}\label{eq:clasediv}
 \left[ \barmoduli[3](6;-2) \right] = 17108\lambda - 1792\delta_{0} - 4396 \delta_{1}\,.
\end{equation}
\par 
Hence by \cite{hamo} (see in particular Corollary~3.95) we have that the class of $\barmoduli[3](6;-2)$ in the Picard group of the moduli space is given by
\begin{equation*}
 \left[ \barmoduli[3](6;-2) \right] = 17108\lambda- 3584\delta_{0} - 4396 \delta_{1}\,.
\end{equation*}
\smallskip
\par
To conclude note that \cite{boissymero} gives that $\omoduli[3](6;-2)$ has one hyperelliptic components and two other components distinguished by parity.  We denote by $\moduli[3]^{\rm hyp}(6;-2)$ the projection of the hyperelliptic component and by $\moduli[3]^{+}(6;-2)$ and $\moduli[3]^{-}(6;-2)$ the projection of the other components.
In the hyperelliptic case the zero and the pole are distinct \Weierstrass points.  Since there are $8$ such points, the class of  $\moduli[3]^{\rm hyp}(6;-2)$ is  $56$ times the class of the hyperelliptic divisor in $\barmoduli[3]$.  According to \cite[Equation 3.165]{hamo}, the class of the hyperelliptic locus is $9\lambda - \delta_{0} - 3\delta_{1}$. Hence  a direct consequence of Theorem~\ref{thm:classeismenosdos} is the following result.
\begin{cor}\label{cor:comp}
 In the rational Picard group of the stack $\barmoduli[3]$ we have
 \begin{equation*}
  \left[ \barmoduli[3]^{+}(6;-2) \right] + \left[ \barmoduli[3]^{-}(6;-2) \right] = 16604\lambda - 1736\delta_{0} - 3750 \delta_{1} \,.
 \end{equation*}
\end{cor}

 \bibliographystyle{alpha}
 \bibliography{biblio} 

\end{document}